\documentclass[a4paper]{NP-IV}
\usepackage[latin1]{inputenc}
\usepackage{amsmath}
\usepackage{amscd}
\usepackage{mathrsfs}
\usepackage[all]{xy}
\usepackage{graphicx}
\usepackage{pstricks} 
\usepackage{amsfonts}   
\usepackage{pict2e}   
\usepackage{makeidx}
\usepackage[colorlinks=true,linkcolor=red,citecolor=red]{hyperref}

\usepackage{amssymb}
\usepackage{latexsym}
\usepackage{bold-extra}

\renewcommand{\O}{\mathscr{O}}

\newcommand{\simto}{\stackrel{\sim}{\to}}

\DeclareMathOperator\Ker{Ker}

\newcommand\Ns{\mathscr{N}}

\newcommand{\scomm}[1]{
{\red \framebox{\framebox{\framebox{
#1}}}}
}

\newcommand{\comm}[1]{
\noindent{\magenta \framebox{\framebox{\framebox{
\begin{minipage}{450pt}#1
\end{minipage}}}}}
}

\newcommand{\RMod}{R\textrm{-Mod}}
\newcommand{\Rmod}{R\textrm{-Mod}}
\newcommand{\Hom}{\mathrm{Hom}}
\newcommand{\LC}{\mathcal{LC}_K}
\newcommand{\Top}{\mathcal{T}_{\O_K}}
\newcommand{\CLC}{\widehat{\mathcal{LC}}_K}

\usepackage{amsthm}
\swapnumbers

\makeatletter
\def\swappedhead#1#2#3{%
  \thmname{#1}\;%
  \thmnumber{\@upn{\the\thm@headfont#2\@ifnotempty{#1}}}%
  \thmnote{\,{\the\thm@notefont(#3)}}{.~}}
\makeatother

\newtheoremstyle{dotless-thm}
  {10pt}
  {10pt}
  {\itshape}
  {}
  {\bfseries}
  {}
  {.0em}
  {}
\theoremstyle{dotless-thm}

\newtheorem{theorem}{\textbf{Theorem}}[subsection]
\newtheorem{thm-intro}{\textbf{\textsc{Theorem}}}
\newtheorem{rk-intro}[thm-intro]{\textbf{\textsc{Remark}}}
\newtheorem{cor-intro}[thm-intro]{\textbf{\textsc{Corollary}}}
\newtheorem{proposition}[theorem]{\textbf{Proposition}}
\newtheorem{lemma}[theorem]{\textbf{Lemma}}
\newtheorem{corollary}[theorem]{\textbf{Corollary}}
\newtheorem{definition}[theorem]{\textbf{Definition}}
\newtheorem{remark}[theorem]{\textbf{Remark}}
\newtheorem{example}[theorem]{\textbf{Example}}

\newtheorem{notation}[theorem]{\textbf{Notation}}
\numberwithin{equation}{section}

\title[Uncountable Mittag-Leffler and applications to locally convex vector spaces]{An uncountable 
Mittag-Leffler condition\\  with an application to\\ ultrametric 
locally convex vector spaces.}

\author{Andrea Pulita}
\email{andrea.pulita@univ-grenoble-alpes.fr}
\address{Université Grenoble Alpes,
Institut Fourier,
CS 40700,
38058 Grenoble cedex 9}

\date{\today}



\begin{abstract}
Mittag-Leffler condition ensures the exactness of the  inverse
limit of short exact sequences indexed on a partially 
ordered set $(I,\leq)$ admitting a 
\emph{countable} cofinal subset.
We extend Mittag-Leffler condition by
relatively relaxing the countability assumption. 
As an application we prove an ultrametric analogous of a 
result of V.P.Palamodov in relation 
with the acyclicity of Frechet 
spaces with respect to the completion functor. 
\end{abstract}

\begin{document}
\maketitle

\begin{center}
Version of \today
\end{center}

\makeatletter
\renewcommand\tableofcontents{%
    \subsection*{\contentsname}%
    \@starttoc{toc}%
    }
\makeatother

%\if{
\begin{small}
\setcounter{tocdepth}{3} \tableofcontents
\end{small}
%}\fi

\setcounter{section}{0}

%\newpage

\section*{\textsc{Introduction}}
\addcontentsline{toc}{section}{\textsc{Introduction}}

In several mathematical theories one encounters objects 
defined as inverse limits. Typically this happens in sheaf 
theory, where the set of global sections of a 
sheaf is the inverse limit of the local ones. 
Analogous structures actually largely appear in several 
theories such as topos theory, linear algebra, 
algebraic geometry, functional analysis and many others. 
%Indeed, in the spirit of 
%\scomm{SGA - ref}, every left exact 
%functor can be considered as a pre-sheaf.
Limits contain crucial information of the 
original systems and it is interesting to study what 
properties are lost in the limit process. 
One of these is \emph{the exactness of short 
exact sequences}.
%, which turns out to be 
%strongly related to the \emph{non emptiness of certain 
%inverse limits of sets}. 
The importance of this property 
is illustrated again by the example of 
sheaves theory, where 
there is an entire cohomology theory 
devoted to ``\emph{measure}'' 
the default of exactness of the global section functor. 
More specifically, we are interested here in a precise 
criterion, originally due to Mittag-Leffler 
\cite[II.19, $N^o5$, Exemple]{Bou-TG}, 
ensuring that the exactness of short exact sequences is 
preserved when passing to the limit. 
Here is the classical Mittag-Leffler 
statement\footnote{Following the tradition, 
we state it for $R$-modules. However, it holds more 
generally for inverse systems of topological groups and 
certain Abelian categories as considered 
in \cite{Roos-4}. 
%In Section \ref{Rk:extension} we extend it to 
%this context.
} 
\begin{thm-intro}[Classical Mittag-Leffler]\label{THM1}
Let $R$ be a ring with unit element and 
let $(I,\leq)$ be a directed\footnote{The word directed means that for all $i,j\in I$ there exists 
$k\in I$ such that $k\geq i$ and $k\geq j$.} partially ordered 
set. 
Let $(\rho^A_{i,j}:A_i \to A_j)_{i,j\in I}$, 
$(\rho^B_{i,j}:B_i \to B_j)_{i,j\in I}$ 
and $(\rho^C_{i,j}:C_i \to C_j)_{i,j\in I}$ 
be three inverse systems of 
left (or right) %unitary\footnote{Unitary means that the 
%scalar multiplication by the unit element of $R$ is the 
%identity map of the $R$-module.} 
$R$-modules indexed on $I$. For all $i\in I$ consider an exact 
sequence $0\to A_i\xrightarrow{g_i}B_i\xrightarrow{h_i} C_i\to 0$ compatible with the 
transition maps of the systems\footnote{i.e. for all $j\geq i\in I$ one has $g_i\circ \rho^A_{j,i}=\rho^B_{j,i}\circ g_j$ and $h_i\circ \rho^B_{j,i}=\rho^C_{j,i}\circ h_j$}. Assume that 
\begin{enumerate}
\item There exists a  cofinal\footnote{A subset 
$J\subseteq I$ is cofinal if for all $i\in I$ there exists 
$j\in J$ such that $j\geq i$.} subset of 
$I$ which is at most countable;
\item For all $i\in I$, there exists $j\geq i$ such that for all 
$r\geq j$ one has 
\begin{equation}\label{eq :ML condition}
\rho^A_{j,i}(A_j)=\rho^A_{r,i}(A_r)\;.
\end{equation}
\end{enumerate}
Then, the short sequence of limits 
\begin{equation}
0\xrightarrow{\;\;\;}\varprojlim_{i\in I}A_i\xrightarrow{\; g\;}\varprojlim_{i\in I}B_i\xrightarrow{\; h\;}\varprojlim_{i\in I}C_i\xrightarrow{\;\;\;} 0
\end{equation}
is exact and the first derived functor $\varprojlim_{i\in I}^{(1)}$ of 
$\varprojlim_{i\in I}$ vanishes at $(A_i)_i$ : 
$\varprojlim_{i\in I}^{(1)}A_i=0$.
\end{thm-intro}
%We recall that the order relation on $I$ gives rise to a topology on the set $I$ such that sheaves of $R$-modules 
%for this topology correspond bijectively to inverse systems 	of $R$-modules indexed on $(I,\leq)$ (cf. \cite[p.4]{Jensen}). 

The condition ii) of the theorem 
is not a necessary condition for the vanishing 
of $\varprojlim_{i\in I}^{(1)}$. Actually,  if $I$ is the set 
of natural numbers $\mathbb{N}=\{1,2,3,\ldots\}$, then 
condition ii) characterizes 
inverse systems $(A_i)_i$ satisfying 
$\lim_{i\in I}^{(1)}A_i\otimes E=0$ for all $R$-module 
$E$ (cf. \cite{EMMANOUIL}).  
%The proof of Theorem \ref{THM1} consists in two steps. 
On the other hand, condition i) is quite restrictive. 
From it, one deduces the existence of a map 
$\tau:\mathbb{N}\to I$ 
respecting the order relation whose image is 
a cofinal subset of $I$ (cf. Lemma 
\ref{Lemma: directed subset N}). 
The existence of $\tau$ is a strong condition because it 
implies that \emph{for all inverse systems} 
$(Q_i)_{i\in I}$  of $R$-modules 
and for all $n\geq 0$ we have a canonical 
isomorphism 
$\varprojlim_{i\in I}^{(n)}Q_i\cong
\varprojlim_{i\in \mathbb{N}}^{(n)}Q_i$ 
between the $n$-th 
derived functors of $\varprojlim$  
(cf. \cite[Theorem B]{Mitchell}).  Hence, 
from a cohomological point of view, 
\emph{inverse systems over $I$ are indistinguishable 
from those  
over $\mathbb{N}$}. 
In particular, the claim implies 
$\varprojlim_{i\in I}^{(n)}A_i=0$, for all integer 
$n\geq 2$, because this is true for every inverse system of modules indexed by $\mathbb{N}$ (cf. 
\cite{Mitchell}, 
see below).

The proof of Mittag-Leffler Theorem deals with 
the surjectivity of the 
map $h$ by a quite explicit set-theoretical argument. Namely, if 
$x=(x_i)_{i\in \mathbb{N}}\in\varprojlim_{i\in \mathbb{N}}C_i$, then the inverse images 
$h_i^{-1}(x_i)\subset S_i$ form an inverse system of 
\emph{sets}, whose inverse limit verifies 
$h^{-1}(x)=\varprojlim_{i\in\mathbb{N}}h_i^{-1}(x_i)$. 
The fact that this limit of \emph{sets} is not empty 
follows from the fact that this system is ``\emph{locally}''\footnote{The word locally here has a precise meaning. It is 
possible to associate to $(I,\leq)$ a topology on $I$ such 
that sheaves on $I$ with respect to this topology are 
exactly inverse systems indexed on $I$. In this 
correspondence, the global sections of a sheaf over $I$ 
is exactly the inverse limit of the associated systems 
(cf. \cite[p.4]{Jensen}, see Section \ref{Section : inv syst and sh}).} 
isomorphic to $(A_i)_{i\in \mathbb{N}}$ 
and condition ii) allows us to 
replace this system by a system of sets  indexed on 
$\mathbb{N}$  with surjective 
transition maps, which 
obviously has a non 
empty inverse limit.

\if{This implies that 
$\varprojlim_{i\in\mathbb{N}}^{(1)}K'_{i}=0$ by an
explicit argument based on the fact for an inverse system 
\emph{of sets} (not $R$-modules) 
indexed on $\mathbb{N}$ with surjective 
maps always has a \emph{non empty} limit 
(cf. \scomm{\ref{}}).
More specifically, the inverse image in 
$(S_i)_{i\in\mathbb{N}}$ of a 
sequence  
$(x_i)_{i\in \mathbb{N}}\in\varprojlim_{i\in \mathbb{N}}T_i$ is an
inverse system \emph{of sets} which is 
``\emph{locally on $\mathbb{N}$}" 
isomorphic to $(S'_i)_{i\in\mathbb{N}}$, in particular it 
has surjective transition morphisms which allows to 
construct step by step a compatible sequence 
$(y_i)_{i\in\mathbb{N}}\in \varprojlim_{i\in\mathbb{N}}S_i
=\varprojlim_{i\in I}S_i$ 
which lies in the inverse image 
of $(x_i)_{i\in\mathbb{N}}$. The map $g$ is then 
surjective.
}\fi
\medskip

In this paper we are interested in extending this statement \emph{relaxing the countability 
condition i) of Theorem \ref{THM1}}.  
The situation is indeed more dangerous because, for 
instance, there are explicit non trivial examples of 
inverse systems \emph{of sets} indexed on some uncountable poset $I$ 
with \emph{surjective transition maps} whose inverse 
limit is empty (cf. \cite[III.94, Exercice 4-d)]{Bou-ENS}), 
so that the last part of the above proof is strongly jeopardized. 
Indeed, without the countability assumption i), there are actually few results in literature 
ensuring the non vanishing of an inverse limit of sets. 
The more important ones seem due to Bourbaki \cite[III.57, $\S7$, N.4, Théorème 1% and III.94, Exercise 5
]{Bou-ENS} and \cite[TG.17, $\S3$, N.5, Théorème 1]{Bou-TG}
and impose strong conditions on the sets and the maps, 
conditions that we can classify as of finiteness in nature. 
For instance, it applies to inverse systems of 
finite sets, finite groups, Artinian modules (cf. \cite[III.60, $\S7$, N.4, Examples]{Bou-ENS}) 
or to compact topological spaces \cite[I.64, $\S9$, $N.6$, Proposition 8]{Bou-TG}.\footnote{See also the more general 
case of linearly compact modules with continuous maps \cite[Théorème 7.1, p.57]{Jensen}.} 
%\cite[II.17, $\S3$, $N.7$, Theorem 1]{Bou-TG} where the maps are supposed to have \emph{dense image} QUESTO E' COUNTABLE;

\if{We notice that they all require some strong conditions on 
the objects. For instance, a consequence of those criteria 
is the well known theorem claiming that the conclusion of 
Theorem \ref{THM1} remains true replacing condition i) 
by the fact that $C_i$ is an 
Artinian $R$-module for all $i\in I$ (cf. \cite{}).}\fi

These issues show that without countability 
assumption on $I$ the first derived functor 
$\varprojlim_{i\in I}^{(1)}A_i$ 
possibly not vanishes for an inverse system with 
surjective transition maps. 
Therefore, several authors addressed 
the question of what can be said about the smallest natural number $s\geq 0$ such that for all $m\geq s$ and all inverse systems $(M_i)_{i\in I}$ one has $\varprojlim_{i\in I}^{(m)}M_i=0$ (this number is called 
\emph{cohomological dimension of the poset $I$}).
Barry Mitchel proved that if $(I,\leq)$ admits a cofinal 
subset of cardinal $\aleph_n$, 
and if $n$ is the smallest natural number 
with this property, 
then for all $k\geq n+2$, the $k$-derived functor 
$\varprojlim_{i\in I}^{(k)}$ vanishes on \emph{every} inverse system of $R$-modules (cf. 
\cite{Mitchell}, extending previous results of J-E. Roos \cite{Roos-1,Roos-2,Roos-3,Roos-4}, 
\cite{Goblot} and Jensen 
\cite[Proposition 6.2, p.53]{Jensen}). On the other hand, it is known that for any given ring $R$, one 
can find a partially ordered set $(I,\leq)$ and an inverse 
system $(M_i)_{i\in I}$ of $R$-modules indexed by $I$ 
such that for all $n\geq 0$ the $n$-th derived limit 
$\varprojlim_{i\in I}^{(n)}M_i$ is not 
zero \cite[Proposition 6.1, p.51]{Jensen}.

In particular, this last result shows that for 
the vanishing of $\lim_{i\in I}^{(1)}A_i$ in Theorem 
\ref{THM1}, 
some \emph{finiteness} condition is really 
needed either on the set 
$I$, or on the objects, or on the transition maps. 
For instance, the countability 
condition i) in Theorem \ref{THM1} can be seen as a 
finiteness assumption on the set $I$ and 
condition ii) is a finiteness condition on the transition 
maps. 
On the other hand, 
the quoted statements of Bourbaki, or their 
consequence for Artinian $R$-modules, 
can be considered 
as finiteness condition on the 
nature of the objects $A_i$.  

Surprisingly enough, if $I$ does not contain any 
cofinal countable subset and if no condition about on $R$  
and the modules $A_i$ are made, 
then in our knowledge no  
statement ensuring the vanishing of 
$\varprojlim_{i\in I}^{(1)}A_i$ exists in literature. 
Nevertheless, in this general context, 
there are interesting cases of inverse 
systems
\if{$(\rho_{i,j}^A:A_i\to A_j)_{i,j\in I}$ of 
$R$-modules indexed on a partially ordered 
set $I$ without countable cofinal subsets, 
whose objects do not have any reasonable finiteness 
property,}\fi  behaving very similarly to Mittag-Leffler 
ones just because much part of the restriction maps 
$\rho_{i,j}^A$ are isomorphisms and their limit is 
then \emph{``controlled'' by some countable 
subset of maps}. 
Situations of this type show up 
for instance in sheaf theory as pull-back of some sheaf 
on a Stain space, which actually inspired our approach to 
this problem. 
In Section \ref{Section : LC} we give an interesting 
example provided by the theory of 
ultrametric locally convex topological vector spaces. We 
prove an ulrametric 
analogous of a result of V.P.Palamodov 
\cite{Palamodov}, in relation with 
the acyclicity of Fréchet spaces 
with respect to the completion functor.
In that case, a direct set-theoretical attempt as in Bourbaki is unhelpful as one can easily see. % and
% the definition of $\varprojlim_{i\in I}^{(1)}A_i$ 
%requires homological algebra.
%, where \v{C}ech complex controls the 
%vanishing of the cohomology groups.

We provide here two generalizations of Theorem 
\ref{THM1} to the case of an uncountable $I$ 
without countable cofinal subsets 
that only involve a finiteness condition on the 
transition maps of the system $(A_i)_{i\in I}$ 
and no conditions on $I$ nor on the objects.
\begin{thm-intro}[cf. Corollary \ref{Cor: THM 2}]\label{THM2}
Let $R$ be a ring with unit element and 
let $(I,\leq)$ be a directed partially ordered set. 
Let $(\rho^A_{i,j}:A_i \to A_j)_{i,j\in I}$
%, $(\rho^B_{i,j}:B_i \to B_j)_{i,j\in I}$ 
%and $(\rho^C_{i,j}:C_i \to C_j)_{i,j\in I}$ 
be an %three 
inverse systems of left (or right) $R$-modules
indexed on $I$. 
%For all $i\in I$ consider an exact 
%sequence 
%$0\to A_i\xrightarrow{f_i}B_i\xrightarrow{g_i} 
%C_i\to 0$ 
%compatible with the transition maps of the systems. 

Assume that there exists another directed 
partially ordered set 
$(J,\leq)$ together with 
an inverse system of %sets 
$R$-modules 
$(\rho^{S}_{i,j}:S_i\to S_j)_{i,j\in J}$ such that
\begin{enumerate}
\item There exists a cofinal directed subset 
$I'\subseteq I$, a cofinal directed subset $J'\subseteq J$
and a surjective map preserving the order relation
\begin{equation}
p\;:\;I'\to J'\;;
\end{equation} 
\item There exists a system of 
%set-theoretic bijections
$R$-linear isomorphisms 
$(\psi_i:A_i\simto 
S_{p(i)})_{i\in I'}$ such that for all $i,j\in I'$ with 
$i\geq j$ one has a commutative diagram %of sets
\begin{equation}
\xymatrix{
A_{i}\ar@{}[dr]|{\circlearrowleft}\ar[d]_{\rho_{i,j}^{A}}\ar[r]^{\psi_{i}}_{\sim}&
S_{p(i)}
\ar[d]^{\rho_{p(i),p(j)}^{S}}\\
A_{j}\ar[r]_{\sim}^{\psi_{j}}&S_{p(j)}}
\end{equation}
\end{enumerate}
Then, for all integer $n\geq 0$, we have a canonical isomorphism
\begin{equation}
\varprojlim_{i\in I}{\!}^{(n)}A_i
\;\xrightarrow{\;\sim\;}\;
\varprojlim_{j\in J}{\!}^{(n)}S_j.
\end{equation}
In particular, if the partially ordered set $J$ and the 
system $(S_j)_{j\in J}$ satisfy 
the conditions i) and ii) of Theorem 
\ref{THM1} respectively, 
then $\varprojlim_{i\in I}^{(n)}A_i=0$ for 
all $n\geq 1$.
\end{thm-intro}
\begin{thm-intro}[cf. Corollary \ref{Cor : THM 3}]
\label{THM3}
Let $R$ be a ring with unit element and 
let $(I,\leq)$ be a %directed 
partially ordered set. 
Let $(\rho^A_{i,j}:A_i \to A_j)_{i,j\in I}$
%, $(\rho^B_{i,j}:B_i \to B_j)_{i,j\in I}$ 
%and $(\rho^C_{i,j}:C_i \to C_j)_{i,j\in I}$ 
be an %three 
inverse systems of left (or right) $R$-modules
indexed on $I$. 
%For all $i\in I$ consider an exact 
%sequence 
%$0\to A_i\xrightarrow{f_i}B_i\xrightarrow{g_i} 
%C_i\to 0$ 
%compatible with the transition maps of the systems. 

Assume that there exists a \emph{directed} 
partially ordered set 
$(J,\leq)$ together with 
an inverse system of %sets 
$R$-modules 
$(\rho^{T}_{i,j}:T_i\to T_j)_{i,j\in J}$ such that
\begin{enumerate}
\item There exists a cofinal directed subset 
$I'\subseteq I$, a cofinal directed subset $J'\subseteq J$
and a map  preserving the 
order relation 
\begin{equation}
q\;:\;J'\to I'\;
\end{equation}
such that for all $i\in I'$, the set 
$U_i:=\{j\in J',q(j)\leq i\}$, endowed with the partial order induced by $J'$, satisfies at least one of the following 
conditions
\begin{enumerate}
\item $U_i$ is empty;
\item $U_i$ has a unique maximal element $r(i)$;
\item $U_i$ is directed, it has countable cofinal directed 
poset $J_i'$ and the system $(\rho^T_{j,k}:T_j\to T_k)_{j,k\in J_i'}$ satisfies \eqref{eq :ML condition}.
\end{enumerate}
\item For all $i\in I'$ 
there exists an $R$-linear isomorphisms\footnote{Notice that under condition $(a)$ we have $\varprojlim_{j\in U_i}T_{j}=0$, and under condition $(b)$ we have $\varprojlim_{j\in U_i}T_{j}=T_{r(i)}$.} 
$\phi_i:A_i\simto \varprojlim_{j\in U_i}T_{j}$ 
and  for all $k,i\in I'$ with 
$k\geq i$ one has a commutative diagram 
\begin{equation}\label{eq: diagram THM3}
\xymatrix{
A_{k}\ar@{}[drr]|{\circlearrowleft}\ar[d]_-{\rho_{k,i}^{A}}
\ar[rr]^-{\phi_{k}}_-{\sim}&&
\varprojlim_{j\in U_k}T_{j}
%T_{r(k)}
\ar[d]^{\rho^{q_*T}_{k,i}}\\
A_{i}\ar[rr]_-{\sim}^-{\phi_{i}}&&
\varprojlim_{j\in U_i}T_{j}
%T_{r(i)}
}
\end{equation}
where the right hand vertical arrow $\rho^{q_*T}_{k,i}$ 
is deduced by the universal properties of the 
limits.
\end{enumerate}
Then, for all integer $n\geq 0$, we have a canonical isomorphism
\begin{equation}
\varprojlim_{i\in I}{\!}^{(n)}A_i
\;\xrightarrow{\;\sim\;}\;
\varprojlim_{j\in J}{\!}^{(n)}T_j.
\end{equation}
In particular, if the partially ordered set $J$ and the 
system $(T_j)_{j\in J}$ satisfy 
the conditions i) and ii) of Theorem 
\ref{THM1} respectively, 
then $\varprojlim_{i\in I}^{(n)}A_i=0$ for 
all $n\geq 1$.
\end{thm-intro}
Remark that if $J'=\mathbb{N}$ the assumptions of 
Theorem \ref{THM3} are particularly easy.

It is not hard to see that the 
assumptions of Theorem \ref{THM1} imply those of 
Theorems \ref{THM2} and \ref{THM3}.
%(cf. \scomm{REF}), 
Therefore, they 
are both generalizations of Theorem \ref{THM1}. 
Indeed, if $I'\subseteq I$ is a
countable cofinal directed subset 
in Theorem \ref{THM1}, 
then the setting $(I'=J, p=q=id, S_i=A_i=T_i, 
\rho_{i,j}^{S}=\rho_{i,j}^{T}=\rho_{i,j}^{A}, \psi_i=\phi_i=id)$ 
satisfies the assumptions of 
Theorems \ref{THM2} and \ref{THM3}. 
Besides, it is clear that Theorems \ref{THM2} and \ref{THM3} allow the 
set $(I,\leq)$ to be arbitrarily large, while Theorem 
\ref{THM1} artificially forces it to be relatively small.

The proofs of these results rely on the fact that 
inverse systems indexed on $(I,\leq)$ 
can be seen as sheaves on a
topological space $X(I)$ canonically
associated to $(I,\leq)$. 
In this correspondence, inverse limits 
%$\varprojlim_{i\in I}(-)$ 
and their cohomology 
functors $\varprojlim_{i\in I}^{(n)}(-)$ 
coincide with sheaf cohomology groups $H^n(X(I),-)$.  
This coincidence of theories permits to 
apply all sheaf theoretic cohomological 
operations, such as, for instance, pull-back and push-forward.
Indeed, as the reader may recognize, 
condition ii) of Theorem \ref{THM2} 
expresses the idea that the system $(A_i)_{i\in I'}$, 
interpreted as a sheaf on $X(I')$, 
is isomorphic to the \emph{pull-back} 
of the system $(S_j)_{j\in J'}$ by the map 
$p:X(I')\to X(J')$. While in Theorem \ref{THM3}, 
the system $(A_i)_{i\in I'}$ is isomorphic to 
the \emph{push-forward} of $(T_j)_{j\in J}$ by the map 
$q:X(J')\to X(I')$. Actually, Theorem \ref{THM3} is 
a special case of a more general statement that 
holds for possibly non directed partially ordered sets 
and which does not assume specific conditions on $U_j$ (cf. Proposition \ref{Prop.: f-lower-star exact with condition fibers}). The fact that we move 
the set of indexes $I$ along pull-back and push-forward
is in contrast with Theorem 
\ref{THM1}, where one fixes the set of indexes once for 
all and there is no cohomological distinction between 
cohomology over $\mathbb{N}$ and over $I$. 
We show indeed that there is no danger in moving $I$ 
because, in this particular context, the pull-back and 
the push-forward operations behave much 
better than in a general topological space. 
Namely, they preserve cohomology under quite 
mild assumptions. 
Informally speaking, even though $X(I)$ is allowed to 
have an enormous amount of open 
subsets, from a cohomological point of view 
it behaves as a relatively tiny space.

\if{Now, an important difference between the two 
statements is, as we noticed, the fact 
that the existence of a cofinal inclusion $\tau:
\mathbb{N}\hookrightarrow I$ in Theorem 
\ref{THM1} is a condition which is so strong that the 
cohomology of \emph{every} inverse system on $I$ 
equals that of its restriction to $\mathbb{N}$. 
The key reason that allows this special situation is that 
the push-forward $\tau_*$ is an exact functor on the 
categories of sheaves (cf. Lemma \ref{} for more 
details). This is where cofinality condition i) of Theorem 
\ref{THM1} plays a major role.
\if{
Another important difference between the two 
statements is, as we noticed, the fact 
that condition i) of Theorem 
\ref{THM1} is so strong that 
it implies 
$\varprojlim_{i\in \mathbb{N}}^{(n)}
Q_i=
\varprojlim_{i\in I}^{(n)}Q_i$, for all inverse 
system $Q=(Q_i)_{i\in I}$ and all integer $n\geq 0$. 
This is due to the fact that the 
cofinality of $\mathbb{N}$ in $I$ implies that the 
pull-back and push-forward operations are both exact 
and so, the canonical morphism $Q\to i_*Q_{|\mathbb{N}}$  

the cohomology of 
\emph{every} inverse system 
$Q:=(Q_i)_{i\in I}$ on $I$ is the 
push-forward of that of 
an inverse system on $\mathbb{N}$. 
}\fi
This no more true in the setting of 
Theorem \ref{THM2}. Indeed, given any non 
cofinal inclusion $i:\mathbb{N}\subset I$, the 
push-forward $i_*$ 
of an inverse system indexed on $\mathbb{N}$ 
does not preserve its cohomology in general. 
In fact, the key idea here is that, 
instead, the pull-back preserves it.}\fi

Finally, we observe that a set-theoretical attempt to the 
proof of Theorems \ref{THM2} and \ref{THM3} in similarity to the quoted 
claims of Bourbaki is not enough powerful to imply these 
results. It is necessary to use \emph{\v{C}ech cohomolgy} of sheaves theory.

\if{
The major obstruction consists in the fact that in the 
proof of Theorem \ref{THM1}, by a simple 
translation, we may find an bijection 
$h_i^{-1}(x_i)\simto A_i$ which shows that the 
inverse system 
$(h_i^{-1}(x_i))_{i\in \mathbb{N}}$ is \emph{locally} 
isomorphic to $(A_i)_{i\in \mathbb{N}}$ on 
$X(\mathbb{N})$.\footnote{In other words, we have an 
isomorphism between the systems 
$(h_i^{-1}(x_i))_{i\in \mathbb{N}}$ and 
$(A_i)_{i\in \mathbb{N}}$ only when we 
restrict them to any subset of the form 
$\{0,1,\ldots, n\}\subset\mathbb{N}$.} 
This is enough to guarantee that a set-theoretic 
version of the Mittag-Leffler 
condition ii) of Theorem \ref{THM1} holds for 
$(h_i^{-1}(x_i))_i$. 
This ensures the non-vanishing of  
$h^{-1}(x)=\varprojlim_{i\in 
\mathbb{N}}h^{-1}_i(x_i)$. }\fi

Although certainly possible, 
an extension of these results to the context of 
inverse limit of \emph{non abelian} 
groups fits in the context of non abelian cohomology of 
sheaves and goes beyond the scopes of this paper. 
Indeed, since this is a result 
that is used by a wide range of mathematicians, 
we made the choice to maintain this paper 
as self contained and basic as possible.

\if{\comm{
lepagine web che mi hanno chiarito la cosa sono le seguenti (vedi TEX file):

%https://mathoverflow.net/questions/36466/do-we-have-non-abelian-sheaf-cohomology

%https://mathoverflow.net/questions/125200/coboundaries-and-gluing-in-cech-cohomology-intuition/125288#125288

}
}\fi

\if{
In down to earth terms, equality 
\eqref{eq: lim=0 general} follows from the 
exactness statement \eqref{eq: 04} and this last is 
proved, as well as Theorem \ref{THM1}, 
using a key 
set-theoretical statement on the non vanishing of 
the inverse system of \emph{sets} 
$(g_i^{-1}(x_i))_{i\in I}$ as mentioned in the 
comments after Theorem \ref{THM1}. 
Imitating Bourbaki, we prove such key statement more 
generally for uniform spaces which allows the 
generalization of Theorem \ref{THM2} to the context of 
topological (possibly non commutative) groups (cf. 
Corollary \ref{} \scomm{Completare}).

We also provide another generalization of Theorem 
\ref{THM1} based on another push-forward 
technique....\scomm{Completare}
}\fi

\subsubsection*{Acknowledgments.}
This work was partially done during a sabbatical term of 
the author at Imperial College of London from January to 
March 2020. The author wishes to thank Imperial College 
of London and UMI Abraham de Moivrefor for the 
hospitality, French CNRS and the ANR 
project ANR-15-IDEX-02 which all partially supported this 
work.

Moreover, the author wishes to thank Stephane Guillermou for useful discussions and helpful comments. 
Thanks also to Michel Brion and Jean-Pierre Demailly for 
having shown interest in this work.

\section{Notations}
We fix once for all a ring $R$ with unit element and 
denote by $\RMod$ the category of left $R$-modules. 
We denote by $\mathcal{S}$ the category of sets.
%By \emph{poset} we mean a partially ordered set. 
Let $\leq$ be an partial order relation on a set  $I$.  
For brevity, we use the terminology 
\emph{poset} for partially ordered set and we may 
indicate $(I,\leq)$ by $I$. For all  $i\in I$ we set 
\begin{eqnarray}
\Lambda(i)&:=&\{j\in I, j\leq i\}\;,\label{eq : Lambda(i)}\\
V(i)&:=&\{j\in I,j\geq i\}\label{eq : V(i)}\;,
\end{eqnarray}
and $D(i)=I- V(i)=\{j\in I, j\notin V(i)\}$. 

We say that the poset $I$ is 
\emph{directed} if for all $i,j\in I$ 
there exists $k\in I$ such that $k\geq i$ and $k\geq j$. 
If $I$ is directed, we say that a subset $S\subset I$ is 
\emph{cofinal} if for all $i\in I$ we can 
find $s\in S$ such that $s\geq i$ (notice that $S$ is 
possibly not directed).
For the moment, we do not assume $I$ directed, 
this condition will be specified when necessary. 

We say that a function 
$f:(I,\leq)\to(J,\leq)$ between two poset 
\emph{preserves the order relations} 
if for all $i, j\in I$ verifying 
$i\leq j$ we have $f(i)\leq f(j)$. 
We also say that $f$ is order preserving.

Everywhere in the paper ``\emph{countable}'' 
means \emph{at most countable} (i.e. finite or in 
bijection with the set of natural numbers $\mathbb{N}$).
\subsection{Inverse systems and inverse limits}

As it is usual, we interpret a poset 
$(I,\leq)$ as a category where 
the objects are the elements of $I$ and, for all $i,j\in I$, 
the set of morphisms $Hom(i,j)$ has exactly 
one element if 
$i\geq j$ or it is empty if either $j<i$ or  $i$ and $j$ are 
not comparable. 

We now introduce the notion of inverse system of sets or 
$R$-modules. An \emph{inverse system in 
of sets indexed on $(I,\leq)$} is a covariant 
functor from 
$(I,\leq)$ to $\mathcal{S}$. We denote by 
$\mathcal{S}^{I}$ the category of inverse 
systems in $\mathcal{S}$ indexed on $(I,\leq)$ (the morphisms being natural transformations of functors).
In down to earth terms, such an inverse system is 
equivalent to 
the datum of a collection $(S_i)_{i\in I}$ of sets 
indexed by $I$, together with a family of maps 
$(\rho^S_{i,j}:S_i\to S_j)_{(i,j)\in I^2, i\geq j}$ such 
that for all $i\in I$ the map $\rho_{i,i}^{S}$ is the 
identity map of $S_i$, and for all $i,j,k\in I$ such that 
$i\geq j\geq k$ one has 
$\rho^S_{j,k}\circ\rho^S_{i,j}=\rho^S_{i,k}$. 
If no confusion is possible, we indicate an inverse system 
as $(S_i)_{i\in I}$ or $S$ and the transition maps by 
$\rho_{i,j}$. 
A morphism $g:(S_i)_{i\in I}\to(T_i)_{i\in I}$ is a 
collection of maps $(g_i:S_i\to T_i)_{i\in I}$ such that for all 
$i,j\in I$ with $i\geq j$ one has 
$\rho_{i,j}^T\circ g_{i}=g_{j}\circ\rho_{i,j}^S$.

An \emph{inverse limit} of an inverse system 
$(S_i)_{i\in I}$ of sets is a set $\widehat{S}$
endowed with a family of maps 
$(\rho^{\widehat{S}}_i:\widehat{S}\to S_i)_{i\in I}$ satisfying 
for all  $i\leq j$ the compatibility relation 
$\rho^{S}_{i,j}\circ\rho^{\widehat{S}}_i=\rho^{\widehat{S}}_j$, which is \emph{universal} 
with respect to this data. In other words,
for all set $S'$ and all family of maps 
$(\rho^{S'}_i:S'\to S_i)_{i\in I}$ satisfying the same 
compatibility relation, 
there exists a \emph{unique} map $\pi:S'\to \widehat{S}$ 
such that for all $i$ one has 
$\rho^{\widehat{S}}_i\circ\pi=\rho^{S'}_i$. 
Therefore, the limit is unique up to a 
unique isomorphism, and we indicate it by 
$\widehat{S}=\varprojlim_{i\in I}(\rho_{i,j}^S:S_i\to S_j)$, or simply 
by $\varprojlim_{i\in I}S_i$ if no confusion is possible. 
The limit $\varprojlim_{i\in I}S_i$ identifies to 
the subset of $\prod_{i\in I}S_i$ formed by sequences 
$(x_i)_{i\in I}\in\prod_{i\in I}S_i$ satisfying for all 
$i\geq j$ the compatibility condition 
$\rho^S_{i,j}(x_i)=x_j$.

If every $S_i$ is an $R$-module and every 
$\rho_{i,j}^S$ is an $R$-module homomorphism, 
we say that $(S_i)_{i\in I}$ 
is a inverse system of $R$-modules. 
Morphisms between inverse systems of $R$-modules are 
morphisms $(g_i)_{i\in I}$ as above where, for all $i\in I$, $g_i$ is an 
$R$-module homomorphism.
We denote by $\Rmod^I$ the category of inverse 
systems of $R$-modules.
The category $\Rmod^I$ inherits almost all the 
properties of $\Rmod$. In particular, it is abelian and it 
has enough injective elements (cf. \cite[p.2]{Jensen}). 
The notion of exactness in $\Rmod^I$ has a particular 
interest for us. A sequence $(A_i)_i\xrightarrow{g}(B_i)_i\xrightarrow{h}(C_i)_i$ 
in $\Rmod^I$ can be seen as a collection of sequences $(A_i\xrightarrow{g_i}B_i\xrightarrow{h_i}C_i)_{i\in I}$ and it 
is \emph{exact} if for all $i\in I$  the sequence $A_i\xrightarrow{g_i}B_i\xrightarrow{h_i}C_i$ is exact in $\Rmod$.
A \emph{short exact sequence} in $(\RMod)^I$ 
is a collection of short exact sequences of $R$-modules
$( 0\to A_i\xrightarrow{g_i}B_i\xrightarrow{h_i}C_i \to 0)_{i\in I}$ indexed on $I$,
such that for every $i, j\in I$ with $i\geq j$ one has 
the compatibility relation $g_j\circ\rho^A_{i,j}=\rho^B_{i,j}\circ g_{i}$ and 
$h_j\circ\rho^B_{i,j}=\rho^C_{i,j}\circ h_{i}$. In other 
words it is an inverse system of exact sequences indexed 
on $(I,\leq)$. It is well known \cite[II.89, $\S6$, N.1, Proposition 1]{Bou-Alg} that such 
a system of short exact sequences gives rise to a left exact 
sequence of limits 
$0\to\varprojlim_{i\in I}A_i\xrightarrow{g}\varprojlim_{i\in I}B_i\xrightarrow{h}\varprojlim_{i\in I}C_i$
and the aim of this paper consists in providing a 
sufficient condition on $(A_i)_i$ ensuring the surjectivity 
of $h$.

\subsection{Sheaves and cohomology}
Let us now recall the notion of sheaf of sets or 
$R$-modules on a topological space.
Let $X$ be a topological space and let $\tau_X$ be the
family of all open subsets of $X$. Let us endow $\tau_X$ 
with the structure of category where the objects are 
open subsets of $X$ and the morphisms 
are just inclusion of them. A \emph{pre-sheaf of sets} 
%(resp. $R$-modules) 
$F$ on $X$ 
is a contravariant functor from $\tau_X$ to the category 
of sets. 
%(resp. $R$-modules). 
Equivalently, a pre-sheaf $F$ is a collection of sets 
%(resp. $R$-modules) 
$(F(U))_{U\in \tau_X}$ and of restriction maps 
%(resp. $R$-linear homomorphisms) 
$(\rho^F_{V,U}:F(V)\to F(U))_{U,V\in\tau_X,U\subset V}$ such that 
for all pair of inclusions 
$U\subset V\subset W$ of open subsets we 
have $\rho_{W,U}^F=\rho_{V,U}^F\circ\rho_{W,V}^F$. 
A \emph{morphism of pre-sheaves} $g:F\to G$ is a morphism of 
functors, i.e. a collection of maps 
$(g_U:F(U)\to G(U))_{U\in \tau_X}$ such that for all 
inclusion $U\subset V$ one has 
$g_U\circ\rho^F_{V,U}=\rho^G_{V,U}\circ g_V$. 
The elements of $F(U)$ are called \emph{sections} of $F$ over 
$U$. Another notation that is typically used to indicate $F(U)$ is $\Gamma(U,F)$.

We say that a pre-sheaf of sets $F$ is a \emph{sheaf} if 
for all open $U$ and all open covering 
$\{U_\alpha\}_{\alpha\in \aleph}$ 
of $U$ it satisfies moreover the following gluing 
properties: i) if $s,t\in F(U)$, and if for all $\alpha$ 
one has 
$\rho_{U,U_\alpha}^F(s)=\rho_{U,U_\alpha}^F(t)$, 
then $s=t$ in $F(U)$; ii)
if $(s_\alpha)_{\alpha\in\aleph}$ 
is a collection of sections $s_\alpha\in U_\alpha$ and if 
for all $\alpha,\beta\in\aleph$ one has 
$\rho^F_{U_\alpha,U_\alpha\cap U_\beta}
(s_\alpha)=\rho^F_{U_\beta,U_\alpha\cap U_\beta}
(s_\beta)$, then there exits a (unique) $s\in F(U)$ such 
that $\rho_{U,U_\alpha}^F(s)=s_\alpha$. 
In other words, if $(V_i)_{i\in I}$ is the covering of $U$ 
whose opens are all possible 
finite intersections of opens of 
$(U_\alpha)_{\alpha\in\aleph}$, 
then $(F(V_i))_{i\in I}$ is an inverse system of sets 
and conditions i) and ii) amount to say that  
$F(U)=\varprojlim_{i\in I}V_i$.
Morphisms of sheaves are just morphisms of 
pre-sheaves. We denote the 
category of sheaves on $X$ by $Sh(X)$. Let $x\in X$, the 
\emph{stalk} of a sheaf $F$  at $x$ is the direct limit 
$F_x:=\varinjlim_{x\in U}F(U)$.

If every $F(U)$ is and $R$-module and every restriction 
map $\rho_{U,V}^F$ is an $R$-linear homomorphism, 
we obtain a \emph{sheaf in $R$-modules}. Morphisms of 
sheaves of $R$-modules are morphisms 
$(g_U)_{U\in\tau_X}$ as above 
such that every $g_U$ is an 
homomorphism of $R$-modules. We denote the 
category of sheaves of $R$-modules on $X$ by 
$\Rmod(X)$. It is an abelian category with enough 
injective objects.  A morphism of sheaves of 
$R$-modules $h:F\to G$ induces, for all $x\in X$, a map 
on the stalk $h_x:F_x\to G_x$, and $h$ is a mono-morphism (resp. 
epimorphism) in $\Rmod(X)$ 
if so is $h_x$, for all $x\in X$. 
A sequence of sheaves of $R$-modules 
$F\to G\to H$ is \emph{exact} 
if for every $x\in X$ so is the 
sequence of stalks $F_x\to G_x \to H_x$. Typically, this 
does not implies the exactness of $F(U)\to S(U)\to H(U)$ 
for all open $U$. The functor 
$\Gamma(X,-):\Rmod(X)\to \RMod$  is left exact and its 
right satellites functors are called the sheaf 
cohomology groups of $F$ denoted by 
$H^n(X,F)=R^n\Gamma(X,F)$ 
(cf. \cite[$\S 4$]{Godement} for the definition). 
Here is a concrete way to compute them. When 
$H^n(X,A)=0$ for all $n\geq 1$, we say that $A$ is an 
\emph{acyclic sheaf} of $R$-module. Then, if 
$A^\bullet:0\to F\to A^0\to A^1\to A^2\to\cdots$ is 
an acyclic resolution (i.e. a long exact sequence of sheaves 
where every sheaf $A^k$ is acyclic), then 
$H^n(X,F)$ can be computed as  
the cohomology groups of the complex of $R$-modules 
$\Gamma(X,A^\bullet):0\to \Gamma(X,A^0)\to 
\Gamma(X,A^1)\to \cdots$. That is, if we set $A^{-1}=0$, 
then for every $n\geq 0$ 
the composite map $\Gamma(X,A^{n-1})\to 
\Gamma(X,A^n)\to \Gamma(X,A^{n+1})$ is zero, 
and if we call $B^n:=B^n(\Gamma(X,A^{\bullet}))\subseteq \Gamma(X,A^n)$ the image of the first map and
$Z^n:=Z^n(\Gamma(X,A^{\bullet}))\subseteq \Gamma(X,A^n)$ the kernel of the second map, then $B^n\subset Z^n$ and we have
\begin{equation}\label{eq: acyclic-1}
H^n(X,F)\;=\;Z^n/B^n\;.
\end{equation}
A standard and compact notation to indicate this process 
consist in writing 
\begin{equation}\label{eq : H^n=R^nGamma}
H^n(X,F)\;=\;R^n\Gamma(X,A^{\bullet})\;.
\end{equation}
%Some classical classes of acyclic sheaves are 
%analyzed in Section 
%\ref{Section : Some acyclicity results}.

\subsection{Topological space associated to a poset.}
\label{Section: Topological space associated to a poset.}
We now define a topological space $X(I)$ 
associated to a poset $(I,\leq)$. The points of $X(I)$ 
are the elements of $I$ 
and open subsets are the subsets 
$U\subseteq I$ with the property that for all $i\in U$ one 
has $\Lambda(i)\subseteq U$ (cf. \eqref{eq : Lambda(i)}). 
%This defines a topology $\tau_{(I,\leq)}$ on $X(I)$.
%A basis of open subsets for $\tau_{(I,\leq)}$ 
%is the family of subsets $(\Lambda(i))_{i\in I}$. 
%A subset $U\subseteq X(I)$ is open if and only if for all 
%$i\in U$ one has $\Lambda(i)\subseteq U$. 
In this topology arbitrary intersections of open subsets 
are open and therefore every subset $S$ of $X(I)$ 
admits a minimum open subset $O(S)=
\cup_{i\in S}\Lambda(i)$ containing it. 
In particular, $\Lambda(i)$ is the smallest open subset 
containing $i$. On the other hand, 
the closure of a subset $S\subset X(I)$ is given by 
$\overline{S}=\cup_{j\in S} V(j)$. 
%A covering 
If $(J,\leq)$ is another poset, then a map 
$f:X(I)\to X(J)$ is continuous if and only if $f$ preserves 
the order relations: if $i\leq j$, then $f(i)\leq f(j)$.
%We denote the category of sheaves in sets (resp. left 
%$R$-modules) over $X(I)$ by $Sh(X)$ (resp. 
%$\RMod(X)$).
The space $X(I)$ acquires special properties  when $I$ is 
a \emph{directed} poset and we will need the following 
Lemma 
\begin{lemma}\label{Lemma: S directed iff Shat is}
Let $S\subseteq I$ be a subset. Then, $S$ is a \emph{directed} 
poset with respect to the order relation induced by $I$ if, 
and only if, so is $O(S)$.%\hfill$\Box$
\end{lemma}
\begin{proof}
Assume that $S$ is directed. Then, for any pair 
$i,j\in O(S)=\cup_{s\in S}\Lambda(s)$ 
there exists $s_i,s_j\in S$ such that 
$i\in\Lambda(s_i)$ and $j\in\Lambda(s_j)$. Therefore, if 
$s\in S$ satisfies $s\geq s_i$ and $s\geq s_j$ we also 
have $s\geq i,j$ which shows that $O(S)$ is 
directed. On the other hand, assume now that 
$O(S)$ is directed. In particular, this implies that for all 
$i,j\in S$ there is $\widehat{s}\in O(S)$ such that 
$\widehat{s}\geq i,j$. Since 
$O(S)=\cup_{s\in S}\Lambda(s)$, 
there is $s\in S$ such that $s\geq\widehat{s}\geq i,j$ 
and the claim follows. 
\end{proof}
\subsection{Inverse systems indexed by $I$ and 
sheaves on $X(I)$.}\label{Section : inv syst and sh}
%Let $\mathcal{C}$ be a category with arbitrary inverse 
%limits and 

Let $(I,\leq)$ be a poset. In this section we recall the 
strong link between the notions of inverse 
systems indexed on $I$ and sheaves on $X(I)$. 
Let $S:=(\rho_{i,j}^S:S_i\to S_j)_{i,j\in I}$ be a inverse 
system of sets indexed on $I$. 
We can define a pre-sheaf $S$ on $X(I)$ 
by associating to every open subset $U$ 
of $X(I)$ the set
$\Gamma(U,S):=\varprojlim_{i\in U}S_i$, 
where $U$ has the order relation induced by $I$. 
For every 
inclusion of open subsets $V\subset U$ there is an 
obvious restriction map $\rho_{U,V}^S:\Gamma(U,S)\to 
\Gamma(V,S)$ provided by the universal property of the 
inverse limit. It is not hard 
to show that $S$ is a 
\emph{sheaf} of sets on $X(I)$ and that
every sheaf on $X(I)$ is of this type. The stalk of a sheaf 
$S$ at a point $i\in X(I)$ is $S(\Lambda(i))$ and it 
coincides with the value $S_i$ at $i$ of the associated 
inverse system.
In the sequel we do 
not distinguish sheaves on $X(I)$ from inverse systems 
and we will indicate them by the same symbol $S$, so 
that we write $S=(S_i)_{i\in I}$, $S(\Lambda(i))=S_i$, or $S(U)=\Gamma(U,S)$.
%However, when it is necessary, we call $\widetilde{S}$ 
%the sheaf associated to the inverse system $S$.
In this correspondence, the inverse limit of an inverse 
system $S=(S_i)_i$ corresponds to the global sections of the 
associated sheaf:
\begin{equation}
\Gamma(X(I),S)\;=\;\varprojlim_{i\in I}S_i\;.
\end{equation}
Moreover, if $(A_i)_i$ is an inverse 
system of $R$-modules, then the derived functors of 
$\varprojlim_{i\in I}^{(n)}A_i$ are defined as the sheaf 
cohomology groups $H^n(X(I),A)$ 
%(cf. \cite[$\S 4$]{Godement} for the definition)
\begin{equation}\label{eq: lim=Hn}
\varprojlim_{i\in I}{}^{(n)}A_i\;:=\;H^n(X(I),A)\;.
\end{equation}
\if{Concretely, if 
$I^\bullet:0\to A\to I^0\to I^1\to I^2\to\cdots$ is 
an acyclic resolution\footnote{A resolution of $A$ 
is just a long exact sequence where the first term is $A$. 
The word acyclic 
means that for all $k\geq 0$ the cohomology groups 
$H^{n}(X(I),I^k)$ are $0$ for every $n\geq 1$.} of $A$, then 
$H^n(X(I),A)$ 
are the cohomology groups of the complex 
$\Gamma(X(I),I^\bullet):0\to \Gamma(X(I),I^0)\to 
\Gamma(X(I),I^1)\to \cdots$, that is 
\begin{equation}\label{eq: acyclic-1}
H^n(X(I),A)\;=\;R^n(\Gamma(X(I),I^\bullet))\;.
\end{equation}}\fi

\subsection{Pull-back and push-forward operations}
Let $(I,\leq)$ and $(J,\leq)$ be two  poset.
Let $f:I\to J$ be a map preserving the order.
Usual pull-back $f^{-1}:Sh(X(J))\to Sh(X(I))$ 
and push-forward $f_*:Sh(X(I))\to Sh(X(J))$ functors 
exist because $f:X(I)\to X(J)$ is just a continuous map 
of topological spaces. We refer to \cite{Godement} 
for their definitions. We bound ourself to describe them 
in term of inverse systems.

\subsubsection{Push-forward.}
Let $S:=(\rho_{i,j}^S:S_i\to S_j)_{i,j\in I}$ be an 
inverse system of sets indexed by $I$ and let $k\in J$. 
By definition, 
for all open subset $U\subseteq X(J)$ the push-forward 
of $S$ is given by $f_*S(U)=S(f^{-1}(U))$ with evident 
transition maps $\rho_{U,V}^{f_*S}=
\rho_{f^{-1}(U),f^{-1}(V)}^{S}$ 
deduced by those of $S$. 
In particular the stalk at a point $k\in J$ is 
given by $(f_*S)_k=f_*S(\Lambda(k))=
\varprojlim_{j\in f^{-1}(\Lambda(k))}S_j$ with 
evident transition maps $\rho^{f_* S}_{k,t}$, 
$k\geq t\in J$, obtained by universal property of the 
limits. Of course, if $S$ is a sheaf in $R$-modules, so is 
$f_*S$.

\subsubsection{Pull-back.}
Let us come now to the pull-back. 
Let now $T=(\rho^T_{i,j}:T_i\to T_j)_{i,j\in J}$ 
be an inverse system of sets indexed indexed by $J$. 
In usual sheaf theory $f^{-1}$ is the sheaf 
associated to the pre-sheaf associating to every open 
$U\subseteq X(I)$ the set 
$\varinjlim_{f(U)\subset V}T(V)$. 
However, in our setting, arbitrary intersections of opens 
are opens, therefore 
$\varinjlim_{f(U)\subset V}T(V)=T(\;O(f(U))\;)$, 
where $O(f(U))=\cup_{i\in U}\Lambda(f(i))$. 
It is indeed easier to define $f^{-1}T$ as an inverse system indexed by $I$. Namely, for every 
$i\in I$, we have $(f^{-1}T)_i:=T_{f(i)}$ and for all 
$i,j\in I$, $i\geq j$, we have
$\rho_{i,j}^{f^{-1}T}:=\rho_{f(i),f(j)}^T$. 
The stalk of $f^{-1}T_i$ is then 
$T_{f(i)}$. Again, when $T$ is a sheaf of 
$R$-modules, so is $f^{-1}T$. 

If $I$ is a subset of $J$ with the order relation induced 
by $J$ and if $f:I\to J$ is the inclusion, we use the 
notation $T_{|I}:=f^{-1}T$.\footnote{Notice that, 
when using this notation, the partial order relation of 
$I$ has to be \emph{induced by that of $J$}. The reason 
is that the injectivity of $f$ is not enough to ensure good 
relations between $\Gamma(X(I),f^{-1}F)$ 
and $\Gamma(X(J),F)$.
 For example, 
assume that we have the set-theoretic equality 
$I=\{i_1,i_2\}=J$ but $i_1$ and $i_2$ are not 
comparable in $I$ while $i_1\leq i_2$ in $J$. Then the 
identity $\iota:I\to J$ preserves the order relation and it 
hence continuous, in this case we do not want to write
$F_{|I}=\iota^{-1}F$.}

%In the following Lemma the fact that the posets are 
%directed is crucial. 
\begin{lemma}\label{Lemma: cofinal pull-back}
Let $f:I\to J$ be a map of directed posets that preserves 
the order relations. Assume that 
the image $f(I)$ is a cofinal subset of $J$. Then 
$$\Gamma(X(J),-)\;\cong\;\Gamma(X(I),-)\circ f^{-1}.$$ 
In other words, for all inverse system 
$T:=(T_j)_{j\in J}$ 
the natural map $\varprojlim_{j\in J}T_j\to
\varprojlim_{i\in I}(f^{-1}T)_i$ is bijective.
\end{lemma}
\begin{proof}
\if{
A straightforward verification with 
compatible sequences gives the result. Notice that the 
fact that the posets are directed is a crucial assumption
for the surjectivity of $\varprojlim_{j\in J}T_j\to
\varprojlim_{i\in I}(f^{-1}T)_i$.
}\fi

%\if{....

The image $J':=f(I)$ is a directed poset 
which is a cofinal subset of $J$. 
We may split $f$ as $f=f_1\circ f_2$, where $f_2:I\to J'$ 
is a surjective map and $f_1:J'\hookrightarrow J$ is an 
inclusion of posets. 
By \cite[III.55, $\S7$, N.2, Prop.3]{Bou-ENS}, we have 
$\Gamma(X(J'),-)\circ f_1^{-1}=\Gamma(X(J),-)$, 
therefore we can assume 
$J=J'$ and $f=f_2$ surjective. 
Let $F=(F_j)_{j\in J}$ be an inverse system of 
$R$-modules indexed by $J$. Then, by definition of 
$f^{-1}$, for all $i\in I$ we have an equality 
$(f^{-1}F)_i= F_{f(i)}$ and the 
natural map $\phi:\varprojlim_{j\in J}F_j\to 
\varprojlim_{i\in I}(f^{-1}F)_i$ 
associates to a compatible 
sequence $x=(x_j)_{j\in J}$ the sequence 
$(x_{f(i)})_{i\in I}$ which is compatible by construction. 
If $\phi(x)=0$, then $x_{f(i)}=0$ for all $i\in I$ and 
the surjectivity of $f$ implies that $x_{j}=0$ for all 
$j\in J$. That is $x=0$ and $\phi$ is injective. 
On the other hand, let us consider $y=(y_i)_{i\in I}\in 
\varprojlim_{i\in I}(f^{-1}F)_i$. For all $j\in J$ the 
inverse image $f^{-1}(j)$ is not empty, and if 
$i_1,i_2\in f^{-1}(j)$, then $y_{i_1}=y_{i_2}$. 
Indeed, since $I$ is directed, there is $i_3\geq i_1,i_2$ 
and for $k=1,2$ we have 
$y_k=\rho^{f^{-1}F}_{i_3,i_k}(y_{i_3})=
\rho^{F}_{f(i_3),j}(y_{i_3})$. Therefore, for all $j\in J$ 
we can chose $i\in f^{-1}(j)$ and set $x_j:=y_{i}$. 
This is independent on the choice of $i$ in $f^{-1}(j)$. 
The sequence $x:=(x_j)_{j\in J}$ is visibly compatible and $\phi(x)=y$.
%}\fi
\end{proof}

\if{
\begin{remark}
In standard sheaf theory, there are functors 
$j_!$ and $j^!$ 
\end{remark}
}\fi
\if{
....

....

\comm{Spostare dopo il Lemma...}
In case $f$ is the inclusion of a directed cofinal subset, 
we have the following well-known result.
\begin{theorem}[\protect{\cite[Theorem B]{Mitchell}}]
Assume that $(J,\leq)$ is a directed poset and 
$I\subseteq J$ a cofinal subset which is directed with 
respect to the order relation induced by that of $J$. 
Denote by $i:I\hookrightarrow J$ the 
inclusion. Then for every inverse system of $R$-modules 
$(A_j)_{j\in J}$, and every integer $n\geq 0$, one has 
\begin{equation}
\varprojlim_{i\in J}{}^{(n)}A_{i}
\;=\;
\varprojlim_{i\in I}{}^{(n)}A_{i}\;.
\end{equation}
In other words $H^n(X(J),A)=H^n(X(I),i^{-1}A)$.
\end{theorem}
}\fi

\subsubsection{Usual properties of $f^{-1}$ and $f_*$.}
\label{Usual properties}
By the above descriptions, it is not hard to see that the 
functor $f^{-1}:\RMod(X(J))\to\RMod(X(I))$ is 
\emph{exact} and $f_*:\RMod(X(I))\to\RMod(X(J))$ is 
 \emph{left exact}. On the other hand, it is well 
known that $f^{-1}$ is \emph{left adjoint} to $f_*$, i.e. 
for all pair of sheaves $S\in \Rmod(X(I))$ and
$T\in \Rmod(X(J))$ there is a canonical functorial 
isomorphism $\Hom_{\Rmod(X(I))}(f^{-1}T,S)\simto
\Hom_{\Rmod(X(J))}(T,f_*S)$. 
Moreover, we have canonical unit and counit 
morphisms $T\to f_*f^{-1}T$ and $f^{-1}f_*S\to S$ 
respectively. 
In general, if $(F,G)$ is a pair of adjoint functors such 
that $F$ is exact and left adjoint to $G$, 
then $G$ sends injective into injective. In 
particular, this is the case of $f_*$ which preserves 
injective objects. It is not hard to see that $f_*$ also 
preserve flabbiness (cf. Section \ref{Section : flabby}).

A typical application of this fact
is the following interpretation of the cohomology groups 
$H^n(X(I),-)$.
Let us denote by $\bullet$ the poset with an individual 
element. The category of sheafs in sets 
(resp. $R$-modules) 
over $X(\bullet)$ is identified with the category of sets  
(resp. $R$-modules) it self by the global functor 
$\Gamma(X(\bullet),-):\RMod(X(\bullet))\simto\RMod$.
The poset $\bullet$ is the final object of the category of 
posets and we denote by
$\pi_I:X(I)\to X(\bullet)$ the projection. Then 
%for every sheaf $S$ in $R$-modules over $X(I)$ 
one has an 
equality of functors 
$\Gamma(X(I),-)=\Gamma(X(\bullet),-)\circ (\pi_I)_*$. By 
the above identification, usually we drop the notation 
$\Gamma(X(\bullet),-)$ and we simply write 
\begin{equation}
\Gamma(X(I),-)\;=\;\pi_{I,*}\;.%\;=\;\varprojlim_{i\in I}
\end{equation}
If $F$ is a sheaf in $R$-modules over $X(I)$ 
we can translate \eqref{eq : H^n=R^nGamma} into the notation
\begin{equation}\label{eq: RnpiI}
H^n(X(I),F)\;=\;R^n\pi_{I,*}(F)\;,
\end{equation}
where $R\pi_{I,*}$ denotes the derived functor of 
$\pi_{I,*}$.

Unfortunately, in general 
$f^{-1}$ does not preserve injectives nor 
any kind of acyclicity and 
for this reason it does not behave well for the 
computation of the cohomology of sheaves. 
Similarly, $f_*$ is not exact and this makes difficult its 
use in the computation of the cohomology because some 
spectral sequences are needed.
However, we provide in 
the next sections some interesting situations where  
\emph{$f^{-1}$ and $f_*$ preserve the cohomology groups}.

\section{Some acyclicity results}
\label{Section : Some acyclicity results}
Let $(I,\leq)$ a poset. In this section we introduce 
several types of acyclic sheaves 
that can be used to compute 
the derived functor of the inverse limit by means of  
\eqref{eq : H^n=R^nGamma}, \eqref{eq: lim=Hn} and 
\eqref{eq: RnpiI}.

\subsection{Flabby and skyscraper sheaves}\label{Section : flabby}
A sheaf $F$ of $R$-modules on $X(I)$ 
is \emph{flabby} if for every open subset 
$U\subseteq X(I)$ the restriction $F(X(I))\to F(U)$ is 
set theoretically surjective. Flabby sheaves are acyclic  
(cf. \cite[Théorème 4.7.1]{Godement}). 
It follows  from the definition that if 
$f:X(I)\to X(J)$ is any continuous map, and if $F$ is a 
flabby sheaf on $X(I)$, then its push-forward $f_*F$  
is flabby. 
This is a simple way to construct acyclic sheaves.

\label{Section : Skyscraper sheaves}
In particular, assume that $I=\bullet$ is a point and  
consider the map $\sigma_j:X(\bullet)\to X(J)$ whose 
image is a point $j\in J$, then for all 
$R$-module $A\in\Rmod=\Rmod(X(\bullet))$, the 
push-forward $\sigma_{j,*}(A)$ is flabby. 
The sheaf $\sigma_{j,*}(A)$ is called the 
\emph{skyscraper} 
sheaf at $j$ with value $A$. It easily seen that for 
$k\in J$ we have $\sigma_{j,*}(A)_k=A$, if 
$k\in V(j)$, and $\sigma_{j,*}(A)_k=0$ otherwise, and 
the transition maps $\rho_{k,t}^{\sigma_{j,*}(A)}$ are 
either the identity maps if $k\geq t\in V(j)$,  
or they equals $0$ otherwise. Skyscraper sheaves are 
acyclic because $\sigma_{j,*}$ preserves flabbiness.

\subsection{Godement resolution}\label{Section : Godement}
We now use skyscraper sheaves to define an acyclic 
resolution of every sheaf of $F$ of $R$-modules over 
$X(J)$ called the \emph{Godement resolution of $F$}. 
We maintain the notation of Section \ref{Section : Skyscraper sheaves}. By adjunction, for all $j\in J$, 
we have a canonical morphism
$F\to \sigma_{j,*}\sigma_j^{-1}F$. 
Therefore, we have a morphism 
$\sigma^F:F\to 
\prod_{j\in J}\sigma_{j,*}\sigma_j^{-1}F$. 
Let us call $Gode(F)$ this product of sheaves. 
Then $Gode(F)$ is flabby because 
skyscraper sheaves are flabby
and a product of flabby sheaves is flabby.

One sees that the sheaf
$G=Gode(F)$ is associated with the 
inverse system $(G_j)_{j\in J}$ defined as 
$G_j:=\prod_{k\in\Lambda(j)}F_k$ with 
restriction maps $\rho_{j,t}^G$ given by 
canonical projections between products. 

The map $\sigma^F:F\to Gode(F)$ is a mono-morphism, 
indeed for every $j\in J$ its stalk at $j$ is the map 
$\sigma_j^F:F_j\to G_j=\prod_{k\in\Lambda(j)}F_k$, that is 
the product $\sigma^F_j=\prod_{k\leq j}\rho_{j,k}^F$. The injectivity then follows from the fact that 
$\rho_{j,j}^F$ is the identity. 
Now, we may consider the quotient 
$Gode(F)/F$ and include it into its $Gode(Gode(F)/\sigma^F(F))$ 
and repeating inductively this operation we obtain a 
flabby resolution $0\to F\to G^0\to G^1\to\cdots$ of 
$F$ which is called the 
\emph{Godement resolution} of $F$. 

\subsection{Directed posets and weak flabbiness}
Flabbiness is not really a 
common property because, 
for instance, if we have two disjoint open 
subsets $U$ and $V$ of $X(I)$, then 
$F(U\cup V)=F(U)\times F(V)$  and the surjectivity of 
$F(X(I))\to F(U)\times F(V)$ tells us that 
\emph{any arbitrary  
pair of sections over $U$ and $V$ have to 
glue to a global section over $X(I)$}. In particular, a 
constant sheaf is possibly not flabby (cf. Section \ref{Section:Acyclicity of constant sheaves}).
This problem related 
to connectedness is avoided with the introduction of 
a weaker notion, due to C.U. Jensen, called 
weak flabbiness in the context of \emph{directed} posets 
which is satisfied by a larger class of sheaves over $X(I)$ 
and is more suitable for our purposes. 
\begin{definition} \label{Def : weak-flabbines}
Let $(I,\leq)$ be a poset. We say that a 
sheaf of $R$-modules $F$ is \emph{weakly flabby} if for every  
open and \emph{directed} subset $J\subseteq I$ the restriction 
$F(X(I))\to F(X(J))$ is surjective. 
\end{definition}
This definition is important when $I$ is a 
\emph{directed} poset because of the following Theorem
\begin{theorem}[\protect{\cite[Théorème 1.8, p.9]{Jensen}}] \label{THM: weak flabby acyclic}
Assume that $(I,\leq)$ is a directed poset. 
Then any weakly flabby sheaf on $X(I)$ is acyclic.
\hfill$\Box$
\end{theorem}
\begin{remark}\label{Remark: every directed wf}
Let $I$ be a directed poset. It was proved by C.U.Jensen 
that, if $F$ is regarded as an inverse system, 
then $F$ is weakly flabby on 
$I$ if, and only if, for any subset 
$J\subseteq I$ which is directed with respect to the 
partial order induced by $I$,
the restriction $F(X(I))=\varprojlim_{i\in I}F_i\to 
\varprojlim_{j\in J}F_j$ is surjective
(cf. 
\cite[Lemme 1.3, p.6]{Jensen}). That is, the open 
condition in Definition \ref{Def : weak-flabbines} can be 
relaxed if needed.
\if{
With an abuse, let us denote by $F(X(J)):=\varprojlim_{j\in J}F_j$.
In particular, 
if $F$ is weakly flabby on $I$, 
then for every $I''\subseteq I'\subseteq I$, where $I',I''$ 
are subsets that are directed with respect to the partial 
order relation induced by $I$, 
the restriction map $F(I')\to F(I'')$ is 
surjective.
}\fi
\end{remark}

\subsection{Acyclicity of constant sheaves over directed 
posets}\label{Section:Acyclicity of constant sheaves}
Another class of interesting sheaves of $R$-modules on 
$X(I)$ is given by constant sheaves. If $C\in\Rmod$ 
is an $R$-module, and if $\pi_I:X(I)\to X(\bullet)$ is the 
constant function 
considered in section \ref{Usual properties}, 
then the \emph{constant sheaf on $X(I)$ with value 
$C\in\Rmod$} 
is defined as $\pi_{I}^{-1}(C)$. For general $X(I)$, 
constant sheaves are not flabby nor 
acyclic and their cohmology groups contain 
important information about the topological space 
$X(I)$. However, if $I$ is a 
\emph{directed} poset, the following corollary 
shows that they are acyclic.
\begin{proposition}\label{Prop: Constant is acyclic}
If $I$ is a directed poset, then any constant sheaf over 
$X(I)$ is weakly flabby, hence acyclic by Theorem 
\ref{THM: weak flabby acyclic}.
\end{proposition}
\begin{proof}
By definition, the inverse system $(C_i)_i=\pi_{I}^{-1}(C)$ satisfies $C_i=C$ for 
all $i\in I$ and the transition maps 
$\rho_{i,j}^{\pi_{I}^{-1}(C)}$ are the identities. Since 
$I$ is directed, we have $\varprojlim_{i\in I}C_i=C$ and 
the same holds for any directed poset $J\subset I$. 
The 
claim then follows from 
Jensen Theorem \ref{THM: weak flabby acyclic}.
\end{proof}
\subsection{Inverse image and weakly flabbiness}
\label{Section:Inverse image and weakly flabbiness}
In a general topological space the inverse image functor 
does not preserve flabbiness. 
However, in our context, weak flabbiness is preserved when we have directed posets.
\begin{proposition}
\label{Prop: inverse image of weakly flabby}
Let $f:I\to J$ be a map of directed posets that preserves 
the order relations. If $W$ is a weakly flabby sheaf on 
$J$, then $f^{-1}W$ is weakly flabby.
\end{proposition}
\begin{proof}
Let $I'\subseteq I$ be a directed subset of $I$. 
We consider $f(I')\subseteq f(I)$ as 
subsets of $J$ with the order relation induced by $J$. 
They are both directed poset. They are possibly not open 
in $J$. However, with an abuse, let us set $W(X(f(I)))=
\varprojlim_{j\in f(I)}W_j$ and similarly for 
$W(X(f(I')))$. Since $W$ is weakly flabby, both 
restrictions $W(X(J))\to W(X(f(I)))$ and 
$W(X(J))\to W(X(f(I')))$ are surjective by Remark \ref{Remark: every directed wf}. 
Hence, so is the restriction map $W(X(f(I)))\to 
W(X(f(I')))$ by composition. Now, by Lemma \ref{Lemma: cofinal pull-back} the 
restriction $f^{-1}W(X(I))\to f^{-1}W(X(I'))$ equals 
the restriction $W(X(f(I)))\to W(X(f(I')))$.
The claim follows.
\end{proof}
In the proof of Proposition 
\ref{Prop: inverse image of weakly flabby}
a key ingredient is Lemma 
\ref{Lemma: cofinal pull-back} in which the 
fact that the posets are directed is a crucial assumption.
The following proposition is a similar statement for 
possibly not directed posets.

\begin{proposition}
\label{Proposition: godement pull-back}
Let $f:I\to J$ be a map of posets that preserves 
the order relations. 
Assume moreover that $I$ is directed. 
Then the following hold:
\begin{enumerate}
\item Let $A$ be a skyscraper sheaf on $X(J)$, then 
the inverse image $f^{-1}A$ of $A$ is weakly flabby.
\item Let $F$ be a sheaf of $R$-modules over $J$ 
and let $Gode(F)$ be the Godement sheaf associated to 
$F$. Then $f^{-1}(Gode(F))$ is wealkly flabby.
\item The inverse image of the Godement resolution of 
$F$ is a weakly flabby resolution of $f^{-1}(F)$.
\end{enumerate}
\end{proposition}
\begin{proof}
Let $j\in X(J)$ and $A\in \Rmod$. Let us denote by 
$(A_k)_{k\in J}:=\sigma_{j,*}A$ the skyscraper sheaf at 
$j\in X(J)$ with value $A$ (cf. Section \ref{Section : flabby}). 
Since $A$ is flabby on $\{j\}$, 
so is $\sigma_{j,*}A$ on $X(J)$. 
We want to show that 
$F:=f^{-1}(\sigma_{j,*}A)$ is weakly flabby over 
$X(I)$. Let $U\subset X(I)$ be any open subset 
which is directed as a poset with the order relation 
induced by $X(I)$. 
Then, we need to show that 
the restriction map $F(X(I))\to F(U)$ is surjective.
Now, by the definition of $f^{-1}$, this restriction 
map identifies to 
the natural restriction map 
$\rho_{O(f(X(I))),O(f(U))}^{\sigma_{j,*}A}:\sigma_{j,*}A(O(f(X(I))))\to 
\sigma_{j,*}A(O(f(U)))$, which is surjective 
because $\sigma_{j,*}A$ is flabby.
\if{
By Lemma \ref{Lemma: S directed iff Shat is}, both 
$O(f(X(I)))$ and $O(f(U))$ are directed 
posets with respect to the partial order relation induced 
by $J$. Moreover, as in the proof of Proposition 
\ref{Prop: Constant is acyclic}, we have
$\sigma_{j,*}A(O(f(X(I))))=A$ (resp.  
$\sigma_{j,*}A(O(f(U)))=A$)
if $j\in O(f(X(I)))$ (resp. $j\in O(f(U))$), 
and $0$ otherwise. 
Therefore, this restriction map 
is either the identity map or the $0$ map whose target is $0$. 
In both cases the restriction map is surjective, and the claim i) follows.

\scomm{Credo che si semplifichi la parte precedente utilizzando che $\sigma_{j,*}(A)$ è flabby.}
}\fi
Let us now prove ii). By definition $Gode(F)=
\prod_{j\in J}\sigma_{j,*}\sigma_j^{-1}F$. Since 
$f^{-1}$ commutes with products 
and since products of weakly flabby is weakly flabby, it is 
enough to prove that if $S=\sigma_{j,*}A$ is a 
skyscraper sheaf on 
$X(J)$, then $f^{-1}S$ is weakly flabby. The claim then 
follows from i).
The third statement is also an immediate consequence of 
the exactness of $f^{-1}$ and of ii).
\end{proof}

\begin{theorem}\label{Thm: pull-back preserve coh}
Let $f:I\to J$ be a map of directed posets preserving the 
order relations and 
such that $f(I)$ is a cofinal subset of $J$. 
Then for all sheaves of $R$-modules $F$ on $X(J)$ one has
\begin{equation}
H^n(X(I),f^{-1}F)\;=\;
H^n(X(J),F)\;.
\end{equation}
In particular, $F$ is acyclic if, and only if, so is $f^{-1}F$.
\end{theorem}
\begin{proof}
The proof is straightforward. 
We use \eqref{eq: acyclic-1} to compute the cohomology 
of $f^{-1}F$. 
Let $0\to F\to G^0\to G^1\to\cdots$ be the Godement 
resolution of $F$. 
Its pull-back $0\to f^{-1}F\to f^{-1}G^0\to 
f^{-1}G^1\to\cdots$ is a resolution of $f^{-1}F$ because $f^{-1}$ is an
exact functor. Every term of the sequence is acyclic 
by Proposition \ref{Proposition: godement pull-back}. 
Therefore, by \eqref{eq: acyclic-1}, 
we know that the complex 
$0\to \Gamma(X(I),f^{-1}F)\to 
\Gamma(X(I),f^{-1}G^0)\to 
\Gamma(X(I),f^{-1}G^1)\to\cdots$ computes the 
cohomology of $f^{-1}F$. Now, Lemma 
\ref{Lemma: cofinal pull-back} ensures that this 
complex equals $0\to \Gamma(X(J),F)\to 
\Gamma(X(J),G^0)\to 
\Gamma(X(J),G^1)\to\cdots$ and the claim follows.
\end{proof}

Theorem \ref{Thm: pull-back preserve coh} 
holds sometimes for non directed posets as we will see 
in Proposition \ref{LEMMA : f and g} in the case of 
Galois connections between posets. 
Let us now show that Theorem 
\ref{Thm: pull-back preserve coh} 
implies Theorem 
\ref{THM2} quite directly, which we translate in term of 
sheaves. 
\begin{corollary}[Theorem \ref{THM2}]
\label{Cor: THM 2}
Let $I$ and $J$ be directed posets and 
$I'\subseteq I$ and 
$J'\subseteq J$ be cofinal directed posets. 
Let $p:I'\to J'$ be a surjective map preserving 
the order relations. Let $A$ and $S$ be sheafs of 
$R$-modules over $X(I)$ and $X(J)$ respectively.
Assume that the restriction $A_{|I'}$ of $A$ to $X(I')$ is 
isomorphic to the pull-back $p^{-1}(S_{|J'})$ of 
the restriction $S_{|J'}$ of $S$ to $X(J')$
\begin{equation}
\psi\;:\;A_{|I'}\xrightarrow{\;\sim\;} p^{-1}(S_{|J'})\;.
\end{equation}
Then, for every integer $n\geq 0$ one has
\begin{equation}
H^n(X(I),A)\;\cong\;H^n(X(J),S)\;.
\end{equation}
\end{corollary}
\begin{proof}
By Theorem \ref{Thm: pull-back preserve coh} applied to 
the inclusions $I'\to I$ and $J'\to J$ we have  
$H^n(X(I),A)=H^n(X(I'),A_{|I'})$ and 
$H^n(X(J),S)=H^n(X(J'),S_{|J'})$, for all 
integer $n\geq 0$. 
Hence, we can assume $I=I'$ and $J=J'$.
Again, Theorem
%The claim then follows directly from Theorem 
\ref{Thm: pull-back preserve coh} 
then ensures $H^n(X(J),S)=H^n(X(I),f^{-1}S)$ and  
$H^n(X(I),f^{-1}S)\cong H^n(X(I),A)$ because $A\cong f^{-1}S$.
\end{proof}

\section{Direct image and exactness}
As mentioned, the direct image functor $f_*$ 
is not exact in general. Spectral sequences are the 
tools that is necessary to compute the cohomology 
spaces of $f_{*}F$ from those of $F$. 
However, we now provide conditions ensuring that $f_*$ 
preserves the cohomology. Notice that the posets are 
possibly not directed.

\begin{proposition}
\label{Prop.: f-lower-star exact with condition fibers}
Let $f:I\to J$ be an order preserving function between posets. 
Let $F$ be a sheaf of $R$-modules over $X(I)$. 
Assume that for all $j\in J$ the restriction 
$F_{|f^{-1}(\Lambda(j))}$ is acyclic as a sheaf over 
the open $U_j:=f^{-1}(\Lambda(j))$. 
That is, for all integer $n\geq 1$ 
one has 
\begin{equation}\label{eq : HnUjF0}
H^n(U_j,F)\;=\;0\;.
\end{equation}
Then
\begin{enumerate}
\item For all injective (resp. flabby) resolution 
$0\to F\to I^1\to I^2\to \cdots$ of $F$, the 
push-forward $0\to f_*F\to f_*I^1\to f_*I^2\to \cdots$ 
is an injective (resp. flabby) resolution of $f_*F$ (i.e. it 
remains exact).
\item For all integer $n\geq 0$ one has
\begin{equation}
H^n(X(J),f_*F)\;=\;
H^n(X(I),F)\;.
\end{equation}
\end{enumerate}
\end{proposition}
\begin{proof}
Let $0\to F\to I^1\to I^2\to \cdots$ be an injective 
(resp. flabby) 
resolution of $F$. Let us set $F^0:=F$ and, 
for all $k\geq 0$, let $F^{k+1}$ be the cokernel of the 
inclusion of $F^k$ into $I^k$. 
We then have the classical diagram
\begin{eqnarray}\label{eq: sequence long with cross}
\xymatrixrowsep{0.1in}
%\xymatrixcolsep{2in}
\xymatrix{
&&&F^1\ar[rd]&&&&F^3\ar[rd]&&&&\\
0\ar[rr]&&I^0\ar[ur]\ar[rr]&&I^1\ar[rd]\ar[rr]&&I^2\ar[ur]\ar[rr]&&I^3\ar[rr]\ar[rd]&&\cdots\\
&F\ar[ur]&&&&F^2\ar[ur]&&&&F^4\ar[ur]&&
}
\end{eqnarray}
We now apply the functor $f_*$ to this diagram. We 
know that $f_*I^k$ remains 
injective (resp. flabby), hence acyclic. Now 
we claim that $0\to f_*F\to f_*I^0\to f_* I^1\to\cdots$ 
is a resolution of $f_*F$, i.e. this sequence 
is exact. This condition can be checked on the stalks. 
Hence, we have to prove that for all $j\in J$ and 
for all $k\geq 0$, the sequence 
$0\to (f_*F^k)_j\to (f_*I^k)_j\to (f_*F^{k+1})_j\to 0$ 
is exact.
As there is a minimal open subset $\Lambda(j)$ 
containing $j$, then if we set 
$U_j:=f^{-1}(\Lambda(j))$, 
this sequence coincides with the sequence
$0\to F^k(U_j)\to I^k(U_j)\to F^{k+1}(U_j)\to 0$.
In other words we have to show that,  for every 
$k\geq 0$, 
$\Gamma(U_j,-)$ sends the short exact
sequence $0\to F^k\to I^k\to F^{k+1}\to 0$ into an exact one. Let us consider the long exact sequence of 
cohomology groups
\begin{equation}
0\to H^0(U_j,F^k)\to H^0(U_j,I^k)\to H^0(U_j,F^{k+1})
\to H^1(U_j,F^k)\to H^1(U_j,I^k)\to H^1(U_j,F^{k+1})\cdots
\end{equation}
Since $I^k$ is acyclic on $U_j$ and 
we have $H^n(U_j,I^k)=0$ for all $k\geq 0$ and all 
$n\geq 1$. Therefore for all $k\geq 0$ and all 
$n\geq 1$ we have an isomorphism 
\begin{equation}\label{eq : induction HnFk}
H^n(U_j,F^{k+1})\;\simto \; H^{n+1}(U_j,F^k)\;.
\end{equation}
Now, for $k=0$, our assumption gives 
$H^n(U_j,F^0)=0$ for all $n\geq 1$ because $F=F^0$ 
is acyclic on $U_j$. The isomorphism 
\eqref{eq : induction HnFk} ensures by 
induction that $F^k$ is acyclic on $U_j$ for all 
$k\geq 0$. Therefore the sequence 
$0\to f_*F\to f_*I^0\to f_* I^1\to\cdots$ 
is exact and it is an injective (resp. flabby) 
resolution of $f_*F$.

It follows then by \eqref{eq : H^n=R^nGamma} 
that $H^{n}(X(J),f_*F)=
R^n\Gamma(X(J),f_*I^\bullet)$. 
Finally, for all $k\geq 0$, the definition of push-forward 
gives $\Gamma(X(J),f_* I^k)=\Gamma(X(I), I^k)$. Hence the 
sequence $0\to \Gamma(X(J),f_* F)\to \Gamma(X(J),f_* I^0)\to \Gamma(X(J),f_* I^1)\to \cdots$ coincides with 
$0\to \Gamma(X(I),F)\to \Gamma(X(I),I^0)\to\Gamma(X(I),I^1)\to\cdots $ 
which computes the cohomology of $F$ by 
\eqref{eq : H^n=R^nGamma}. 
The claim follows.
\end{proof}

\begin{remark}
In Proposition 
\ref{LEMMA : f and g} we will treat a special 
situation where $f_*$ preserves also 
\emph{weakly-flabby} 
resolutions.
\end{remark}

An interesting case where  Proposition 
\ref{Prop.: f-lower-star exact with condition fibers} 
applies is the following
\begin{theorem}\label{Prop: maximum push-forward}
Let $f:I\to J$ be an order preserving function between posets. 
Let $F$ be a sheaf of $R$-modules over $X(I)$. 
Assume that for every $j\in J$ the set 
$U_j=f^{-1}(\Lambda(j))$ satisfies at least one 
among the following conditions hold:
\begin{enumerate}
\item $U_j$ is empty;
\item $U_j$ 
has a unique maximal element (i.e. it is of the form 
$\Lambda(i)$, for some $i\in I$);
\item $U_j$ is a directed poset admitting a countable 
cofinal directed poset $I_j'$ and the system 
$(A_k)_{k\in I'_j}:=F_{|I'_j}$ satisfies 
Mittag-Leffler condition \eqref{eq :ML condition}.
\end{enumerate}
Then, the conclusions i) and ii) of Proposition \ref{Prop.: f-lower-star exact with condition fibers} hold.
\end{theorem}
\begin{proof}
If i) or iii) hold for $U_j$, 
we know by Theorem \ref{THM1} that 
$F_{|U_j}$ is acyclic and the condition of Proposition 
\ref{Prop.: f-lower-star exact with condition fibers} 
is fulfilled. If ii) holds for $U_j$, then $U_j=\Lambda(i)$ for some $i\in I$. Now, the functor 
$\Gamma(\Lambda(i),-)$ is the fiber functor associating 
to a sheaf $F$ its stalk $F_i$ at $i$. Therefore, it is an 
exact functor and it preserves 
injective resolutions. Hence,
for every sheaf $F$ of $R$-modules over 
$X(I)$ the restriction $F_{|\Lambda(i)}$ is acyclic on $\Lambda(i)$. 
Proposition 
\ref{Prop.: f-lower-star exact with condition fibers} then applies.
\end{proof}
\begin{remark}
It was proved by O.Laudal \cite{LAUDAL} that the 
only posets $U$ over which 
every sheaf of $R$-modules is acyclic 
are those admitting a maximum element (i.e. $U=\Lambda(i)$ for some $i\in U$). 
Therefore, any generalization of 
Theorem \ref{Prop: maximum push-forward} to 
more general maps $f$ requires restrictions on the 
class of sheaves $F$ that we consider, as we did in 
condition i).
For instance, let us assume that for all $j\in J$ the poset
$U_j=f^{-1}(\Lambda(j))$ has only finitely many 
maximal elements. This mens that $U_j$ is a finite union 
of open posets of the form $\Lambda(i)$. In this 
situation it might be interesting to use Mayer-Vietoris 
long exact sequence to obtain combinatoric conditions 
on $F$ ensuring \eqref{eq : HnUjF0}. 
\if{However, for more than $3$ maximal elements in 
$U_j$, 
this seems to reduce drastically 
the class of sheaves for which the proposition applies. 
More seriously, similar ideas should be carried out with 
the use of spectral sequences.}\fi
\end{remark}

From another angle, if we assume that $I$ is directed, 
then it might be interesting to replace it by a cofinal 
directed subset $I'$. 
This operation preserve the cohomology 
groups of $F$ and reduces the size of the sets 
$f^{-1}(\Lambda(j))$ (which makes possibly easier to verify  \eqref{eq : HnUjF0}). However, it should be taken with 
some precaution because it does not preserve 
the push-forward (i.e. $f_*F\neq f_*(F_{|I'})$). The claim is the following.

\begin{corollary}
\label{Prop.: f-lower-star exact with condition fibers directed}
Let $I$ be a directed poset and $F$ a sheaf of 
$R$-modules over $X(I)$. 
Let $I'\subseteq I$ be a directed cofinal subset of $I$ 
and let $f:I'\to J$ be an order preserving function between 
posets such that, for all $j\in J$, the restriction 
$F_{|f^{-1}(\Lambda(j))}$ is acyclic as a sheaf over 
the open subset $U_j':=f^{-1}(\Lambda(j))\subset X(I')$.  
That is, for all integer $n\geq 1$, one has 
$H^n(U_j',F_{|X(I')})=0$. 
In particular, this condition is automatically satisfied 
if one of the conditions i), ii), iii) of Theorem 
\ref{Prop: maximum push-forward} holds 
for $F_{|U_j'}$. 
Then, for all integer $n\geq 0$ one has
\begin{equation}
\qquad\qquad\qquad
\qquad
H^n(X(J),f_*(F_{|X(I')}))\;=\;
H^n(X(I),F)\;.\qquad\qquad\qquad\Box
\end{equation}
\end{corollary}

Another interesting case where Corollary 
\ref{Prop.: f-lower-star exact with condition fibers directed} applies 
is of course given by the poset 
 of natural numbers $\mathbb{N}$, 
where \emph{every bounded} open subset has a 
maximum element. 
We obtain the following corollary. Notice that 
\emph{no cofinality condition is required for the inclusion 
of $f(I)$ in $J$.}

\begin{corollary}[Case of a totally ordered countable poset]
\label{Cor.: f-lower-star N}
Let $I$ be a poset and $F$ a sheaf of 
$R$-modules over $X(I)$. 
Assume that $I$ is directed and has a totally ordered 
cofinal subset $N$ which is at most countable 
(i.e. $N$ is finite or isomorphic to 
$(\mathbb{N},\leq)$).\footnote{By Lemma \ref{Lemma: 
directed subset N}, this is equivalent to the simple 
existence of a cofinal subset in $I$ which is at most 
countable.} 
Let $f:N\to J$ be an order preserving function between 
posets such that, for all $j\in J$  the following 
condition holds
\begin{enumerate}
\item if $U_j:=f^{-1}(\Lambda(j))=N$, then
the restriction of $F_{|N}$ satisfies 
Mittag-Leffler condition \eqref{eq :ML condition}.
\end{enumerate}
Then, for all integer $n\geq 0$ one has
\begin{equation}\label{eq : Hnf*countable}
H^n(X(J),f_*(F_{|X(N)}))\;=\;
H^n(X(I),F)\;.
\end{equation}
In particular, i) is an empty condition if 
for every $j\in J$, there exists $\eta\in N$ such that 
$f(\eta)\not\leq j$ (i.e. $f^{-1}(\Lambda(j))
\neq N$, for all $j\in J$).
\hfill$\Box$
\end{corollary}

For the benefit of the reader we now translate Theorem 
\ref{THM3} in the sheaf language. The role of $I$ and 
$J$ is reversed with respect to the statement in the 
introduction and, even though it is not necessary, we 
assume the posets to be directed in order to allow the 
restriction to a cofinal poset.

\begin{corollary}[\protect{Theorem \ref{THM3}}]
\label{Cor : THM 3}
Let $(J,\leq)$ be a directed poset and $A$ a 
sheaf of $R$-modules over $X(J)$.
Assume that there exists 
a \emph{directed} partially ordered set $(I,\leq)$ 
and a sheaf of $R$-modules $T$ over $X(I)$ such that
\begin{enumerate}
\item There exists a cofinal directed subset 
$J'\subseteq J$, a cofinal directed subset $I'\subseteq I$
and a map $q:I'\to J'$ preserving the 
order relation such that for all $j\in J'$, the set 
$U_j=\{i\in I',q(i)\leq j\}$ is either empty, or it 
has a unique maximal element, 
or it has a countable cofinal directed poset $I'_j$ and 
$T_{|X(I'_j)}$ satisfies Mittag-Leffler condition 
\eqref{eq :ML condition}.
\item We have an $R$-linear isomorphism of sheaves 
$\phi:A_{|J'}\cong q_*T_{|I'}$.
\end{enumerate}
Then, for all integer $n\geq 0$, we have a canonical 
isomorphism
\begin{equation}
H^n(X(J), A)\;\cong\;
H^n(X(I), T)\;.
\end{equation}
In particular, if $T$ is acyclic then so is $A$.
%\hfill$\Box$
\end{corollary}
\begin{proof}
By Theorem \ref{Thm: pull-back preserve coh} applied to 
the inclusions $I'\to I$ and $J'\to J$ we have  
$H^n(X(J),A)=H^n(X(J'),A_{|J'})$ and 
$H^n(X(I),T)=H^n(X(I'),T_{|J'})$, for all 
integer $n\geq 0$. 
Hence, we can assume $I=I'$ and $J=J'$. The claim then 
follows from Proposition \ref{Prop.: f-lower-star exact 
with condition fibers directed}.
\end{proof}

\section{Galois connections.}
\label{Section : Galois connections}
In this section we consider Galois connections between 
posets. This is a particularly lucky situation, 
because the operations of push-foward $f_*$ and the 
pull-back $g^{-1}$ coincide and we automatically have 
the benefits of both operations 
(cf. Proposition \ref{LEMMA : f and g} below). 
We begin by the following Lemma 
\ref{Lemma: directed subset N} which says that 
when we have a countable cofinal subset, we 
automatically have a Galois connection with a convenient 
countable totally ordered subset. 

\begin{lemma}\label{Lemma: directed subset N}
Let $J$ be a directed poset that admits a
countable cofinal subset. Then, there exists a countable 
cofinal subset $N\subset J$ which is \emph{directed} 
and totally ordered. 
The set $N$ is finite if, and only if, $J$ has a maximum 
element (in this case we can chose $N$ equal to the maximum 
element of $J$). Otherwise, $N$ is isomorphic 
to the poset of natural numbers 
$(\mathbb{N},\leq)$. Moreover, if $f:N\to J$ denotes 
the inclusion, then there exists a map 
$g:J\to N$ preserving the order relations and such that 
\begin{enumerate}
\item The map $g\circ f:N\to N$ is the identity map.
\item For all $j\in J$, 
$f^{-1}(\Lambda(j))=\Lambda(g(j))$, that is $g(j)$ is the biggest element of $f^{-1}(\Lambda(j))$.
%\hfill$\Box$
\end{enumerate}
\end{lemma}
%\if{
\begin{proof}
Let $S\subseteq I$ be a countable cofinal subset and let 
$S=\{s_1,s_2,\ldots\}$ be an enumeration of $S$. 
Set $\eta_1:=s_1$ and, for all integer 
$n\geq 2$, chose inductively an 
$\eta_n\in J$ such that $\eta_n\geq \eta_{n-1}$ and $\eta_n\geq s_3$. 
We now have an increasing sequence $(\eta_n)_n$ in 
$J$. Let $N\subset J$ be the set of its values. 
Then $N$ is cofinal in $J$ because $S$ is. 
Clearly $N$ is finite 
and totally ordered if, and only if, the sequence is 
stationary, 
and in this case its maximum is also a maximum of $J$. 
Otherwise, we may find a subsequence 
$(\eta_{n_k})_{k\in\mathbb{N}}$ of $(\eta_n)_n$ 
which is strictly increasing whose underling subset is $N$ 
and the map $k\to \eta_{n_k}$ provides a bijection 
between $\mathbb{N}$ and $N$ preserving the order 
relations. 

Now, as $N$ is cofinal, we have 
$J=\cup_{\eta\in N}\Lambda(\eta)$. Since $N$ is 
discrete and totally ordered, for every $j\in J$ there exists a minimum $\eta_j\in N$ such that $j\in\Lambda(\eta_j)$. 
Therefore, we can define a map $g:J\to N$ as 
$g(i)=\min(\eta\in N, i\in\Lambda(\eta))$. The claim follows.
\end{proof}
%}\fi

Recall that if 
\begin{equation}
f\;:\;I\xrightarrow{\quad}J\;\qquad\textrm{and}
\qquad
g\;:\;J\xrightarrow{\quad}I\;
\end{equation}
%$f:I\to J$ and $g:J\to I$ 
are two maps between 
posets that preserve the order relations, then the 
following conditions are equivalent
\begin{enumerate}
\item For all $i\in I$ and all $j\in J$ one has 
$f(g(j))\leq j$ and $g(f(i))\geq i$;
\item For all $i\in I$ and all $j\in J$ we have $f(i)\leq j$ 
if, and only if, $i\leq g(j)$.
\end{enumerate}
In this case, the pair $(f,g)$ is called a 
\emph{Galois connection} 
between $I$ and $J$. If $I$ and $J$ are seen as 
categories, these conditions express 
the fact that $f$ is a \emph{left adjoint} of $g$ and $g$ 
is a \emph{right adjoint} of $f$.
It is not hard to prove that a map $f:I\to J$, 
respecting the partial order relations, admits a right 
adjoint $g:J\to I$ if, and only if, the following condition 
holds
\begin{enumerate}
\item[iii)] For all $j\in J$, there exists $i_j\in I$ 
such that 
$f^{-1}(\Lambda(j))=\Lambda(i_j)$. 
\end{enumerate}
In this case, $i_j$ is the value of $g$ at $j$, so that for 
all $j\in J$ we have 
\begin{equation}
\label{eq : f^-1 lambda j = lambda g j}
f^{-1}(\Lambda(j))\;=\;\Lambda(g(j))\;.
\end{equation}
In particular, when the right adjoint $g$ exists, 
it is uniquely determined by 
\eqref{eq : f^-1 lambda j = lambda g j}. 
Symmetrically, $g:J\to I$ admits a left adjoint if, and 
only if, for all $i\in I$ there exists $j_i\in J$ such that 
$g^{-1}(V(i))=V(j_i)$ and in this case $f(i)=j_i$. 

%Notice that iii) shows that $f\circ g\circ f=f$ and 
%$g\circ f\circ g= g$.
 
\begin{proposition}
\label{LEMMA : f and g}
\label{Lemma: f*=g-1}
Let $(f,g)$ be a Galois connection as above.
Then
\begin{enumerate}
\item The functors $f_*:Sh(X(I))\to Sh(X(J))$ 
and $g^{-1}:Sh(X(I))\to Sh(X(J))$ coincide.
In particular, for every sheaf $F$ of $R$-modules 
over $X(I)$ we have 
\begin{equation}
f_*F\;=\;g^{-1}F\;.
\end{equation}

\item The conditions of Theorem \ref{Prop: 
maximum push-forward} are fulfilled and for every 
sheaf $F$ of $R$-modules over $X(I)$ the conclusions 
i) and ii) of Proposition \ref{Prop.: f-lower-star exact 
with condition fibers} hold.
\item If $I$ and $J$ are both directed posets, 
then $f_*$ preserves weakly flabbiness. 
In particular, it sends weakly flabby resolutions of 
$F$ into weakly flabby resolutions of $f_*F$.
\end{enumerate}
\end{proposition}
\begin{proof}
Let us see $F$ as an inverse system 
$(\rho_{i,j}^F:F_{i}\to F_k)_{i,k\in I}$. 
Then, by definition, for all $j\in J$ 
both $f_*F$ and $g^{-1}F$ verify 
$(f_*F)_{j}=F_{g(j)}=(g^{-1}F)_j$ and, 
for all $j'\geq j$, one has 
$\rho_{j',j}^{f_*F}=\rho^{F}_{g(j'),g(j)}=
\rho^{g^{-1}F}_{j',j}$. Items i) and ii) follow 
immediately. In particular, $f_*$ is exact. 
To prove iii), it is then enough to show that
if $W$ is a weakly flabby sheaf of $R$-modules over 
$I$, then so is $f_*W$ on $J$. 
Since $f_*W=g^{-1}W$, this follows from Proposition 
\ref{Prop: inverse image of weakly flabby}.
\end{proof}

\begin{remark}
Lemma \ref{Lemma: 
directed subset N} admits the following generalization 
which does not involve any \emph{cofinality condition}.
Let $J$ be a directed poset and 
$f:\mathbb{N}\to J$ be an order preserving map 
satisfying the following condition: 
\begin{itemize}
\item[$\bullet$] 
For all $j\in J$, $f^{-1}(\Lambda(j))\neq 
\mathbb{N}$ (i.e. for all $j\in J$ there exists $n\in\mathbb{N}$ such that 
$f(n)\not\leq j$).
\end{itemize}
Then, by item iii) before \eqref{eq : f^-1 lambda j = 
lambda g j}, $f$ admits a right adjoint 
$g:J\to \mathbb{N}$ and 
Proposition \ref{LEMMA : f and g} applies.
\end{remark}

\section{An application to 
$p$-adic locally convex spaces}\label{Section : LC}

In this section we give an application to ultrametric 
locally convex spaces. It is an ultrametric analogous of a result of V.P.Palamodov \cite{Palamodov}.

An ultrametric absolute value on a field $K$ is a function 
$|.|:K\to\mathbb{R}_{\geq 0}$ verifying $|0|=0$, $|1|
=1$, $|xy|=|x||y|$, and $|x+y|\leq\max(|x|,|y|)$ for all 
$x,y\in K$. 
From now on we assume that the absolute value is non 
trivial (i.e. there exists $x\neq 0$ such that $|x|\neq 1$) 
and that $K$ is 
complete with respect to the topology defined by $|.|$. 
We denote by $\O_K=\{x\in K,|x|\leq 1\}$ its ring 
of integers.

An ultrametric seminorm on a $K$-vector space $V$ is a 
function $u:V\to 
\mathbb{R}_{\geq 0}$ such that for all $r\in K$ and  
$x,y\in V$ one has $u(rx)=|r|u(x)$ and $u(x+y)\leq 
\max(u(x),u(y))$.
A locally convex space over $K$ is a topological 
vector space $V$ whose topology is defined 
by a family of ultrametric semi-norms. Recall that $V$ 
has a basis of open neighborhoods of $0$ formed by 
$\O_K$-submodules, we call them \emph{convex opens}. 

A $K$-linear continuous map $f:V\to W$ between locally 
convex spaces is \emph{strict} 
if the topology induced 
by $W$ on the image of $f$ coincides with the quotient topology of $V$.

\begin{proposition}
Let $f:V\to W$ be a $K$-linear strict map between 
Hausdorff complete locally convex spaces. If the kernel 
of $f$ is a Fréchet space, then the image of $f$ is a 
Hausdorff complete closed subspace of $W$.
\end{proposition}
\begin{proof}
Let $V'$ be the kernel of $f$ and $V''$ its image. 
It is enough to show that $V''$ is Hausdorff  and complete with respect 
to the quotient topology induced by $V$. 
For this it we prove that the strict 
short exact sequence $0\to V'\to V\to V''\to 0$ remains 
strict exact after the Hausdorff-completion operation. 
Indeed, $V'$ and $V$ are already Hausdoff and complete. 
Let $I$ be the family of convex 
neighborhoods of $0$ in $V$. The set $I$ is naturally 
partially ordered by the inclusion of subsets. 
For all $D\in I$, set 
$D':=D\cap V'$ and denote by $D''$ the image of 
$D$ in $V''$. 
\if{For all $D_1\subset D_2 \in I$ we have a morphism 
of short exact sequences
\begin{equation}
\xymatrix{0\ar[r]&V'/D'_1\ar[d]\ar[r]&V/D_1\ar[d]\ar[r]&V''/D_1''\ar[d]\ar[r]&0\\
0\ar[r]&V'/D_2'\ar[r]&V/D_2\ar[r]&V''/D_2''\ar[r]&0}
\end{equation}
where the vertical arrows are the natural projections.}\fi 
The Hausdorff completion of the sequence $0\to V'\to V\to V''\to 0$ 
is then the inverse limit of the sequences $0\to V'/D'\to 
V/D\to V''/D''\to 0$ for $D$ running in $I$. 
Let $J$ be the set of open neighborhoods of $V'$ 
of the form $p(D)=D\cap V'$ with $D\in I$. The map 
$p:I\to J$ is surjective and the inverse system 
$(V'/D')_{D\in I}$ is the pull-back of $(V'/D')_{D'\in J}$ 
by $p:I\to J$. The conditions of 
of Corollary \ref{Cor: THM 2} are fulfilled. 
It follows that for all $n\geq 0$ we have 
$\varprojlim_{D\in I}^{(n)}V'/D'=
\varprojlim_{D'\in J}^{(n)}V'/D'$. 
Now, since $V'$ is Hausdorff and Fréchet, then 
$J$ has a countable cofinal subset $N$. The transaction 
maps being surjective, Theorem \ref{THM1} applies and 
$\varprojlim_{D'\in J}^{(n)}V'/D'=0$ for all $n\geq 1$. The claim follows.
\end{proof}

\bibliographystyle{amsalpha}
\bibliography{2012-NP-III}
\end{document}